\documentclass[a4paper,12pt]{article}
\usepackage{amssymb,amsmath,amsthm,times,mathrsfs,fullpage}
\usepackage{tikz,yfonts,shadow}
\usepackage[framemethod=TikZ]{mdframed}
\usetikzlibrary{scopes,shapes,snakes,arrows,shadows}
\usepackage{hyperref}
\hypersetup{plainpages=True, pdfstartview=FitH, bookmarksopen=true,
pdfpagemode=none, colorlinks=true,linkcolor=blue,citecolor=blue}

\def\le{\leqslant}
\def\ge{\geqslant}
\def\dd#1{{\,\rm d}#1}

\def\ve{\varepsilon}
\def\tr#1{\left\lfloor #1\right\rfloor}
\def\cl#1{\left\lceil #1\right\rceil}
\def\ON{\raisebox{-0.2ex}{\scalebox{1.5}{$\oslash$}}}
\def\OY{\hspace*{-0.2ex}\raisebox{-0.88ex}{\scalebox{2.7}{$\bullet$}}}
\def\OE{\bigcirc}
\def\bfi{\mathbf{i}}

\newtheorem{thm}{Theorem}

\newtheorem{lmm}{Lemma}
\newtheorem{prop}{Proposition}

\begin{document}

\title{Random unfriendly seating arrangement \\
in a dining table}
\author{Hua-Huai Chern\\
    Department of Computer Science \\
    National Taiwan Ocean University \\
    Keelung 202\\
    Taiwan
\and Hsien-Kuei Hwang \\
    Institute of Statistical Science, \\
    Institute of Information Science\\
    Academia Sinica\\
    Taipei 115\\
    Taiwan
\and Tsung-Hsi Tsai \\
    Institute of Statistical Science \\
    Academia Sinica\\
    Taipei 115\\
    Taiwan}

\maketitle

\begin{abstract}
A detailed study is made of the number of occupied seats in an
unfriendly seating scheme with two rows of seats. An unusual
identity is derived for the probability generating function, which
is itself an asymptotic expansion. The identity implies particularly
a local limit theorem with optimal convergence rate. Our approach
relies on the resolution of Riccati equations. We also clarify
some simple yet delicate stochastic dominance relations.
\end{abstract}

\section{Introduction}

Freedman and Shepp formulated the ``unfriendly seating arrangement
problem" in 1962 \cite[Problem 62--3]{Freedman62}:
\begin{quote}
    \emph{There are $n$ seats in a row at a luncheonette and people sit
        down one at a time at random. They are unfriendly and so never
        sit next to one another (no moving over). What is the expected
        number of persons to sit down?}
\end{quote}
Let $Z_n$ denote the number of persons sitting down when no further
customers can sit properly without breaking the restriction of
unfriendliness. Solutions with different degree of precision or
generality were later proposed by many. In particular, Friedman and
Rothman \cite{Friedman64} proved that
\begin{align*}
    \mathbb{E}(Z_n)  &= \sum_{0\le k<n}(n-k)
    \frac{(-2)^k}{(k+1)!} \\
    &= \tfrac{1}{2}\left(1-e^{-2}\right) (n+3) -1
    + O\left(\frac{2^n}{(n+2)!}\right),
\end{align*}
for large $n$. The \emph{factorial} error term here seems
characteristic of sequential models of a similar nature; see, for
example, \eqref{VZn}, \eqref{EXn2} and \eqref{VXn} below and
\cite{DR64}. We will provide a general framework for characterizing
such small errors; see Proposition~\ref{prop-facto} below. In
addition, Friedman and Rothman　\cite{Friedman64} extended the
``degree of unfriendliness'' to any　integer $b\ge 1$, where any two
people have to sit with at least $b$ unoccupied seats between them.
This extension was mentioned to be related to R\'enyi's Parking
Problem and to a discrete parking problem studied by MacKenzie (see
\cite{MacKenzie62}) in which cars of the same length $\ell\ge 2$ are
parked uniformly at random along the curb with $n$ unit parking
spaces. Indeed, the latter problem with $\ell=2$ found its origin in
Flory's 1939 pioneering paper \cite{Flory39} in polymer chemistry,
and was later expanded into generic stochastic models under the name
``random sequential adsorption''; see \cite{Evans93} for a
comprehensive survey and \cite{CAP07,Evans05, Penrose02,Penrose08} a
more recent account.

Due to the simplicity and the usefulness of the model, the same
discrete parking problem was also studied independently under
different guises in applied probability and related areas. Page
\cite{Page59} studied a random pairing model in which $n$ isolated
points are paired randomly by adjacency until only singletons
remain. This model is identical to Flory's monomer-dimer model
\cite{Flory39} (or the discrete parking problem \cite{MacKenzie62}
where each car requires 2-unit parking space). The same model was
also encountered in a few diverse modeling contexts. Let $\zeta_n$
denote the resulting number of pairs when no more adjacent pair can
be formed. Then it is easy to see that
\[
    Z_n \equiv \tfrac12 \zeta_{n+1} \qquad (n\ge0).
\]

In addition to deriving a closed-form expression for the first three
moments of $\zeta_n$, Page \cite{Page59} also computed the variance,
which, when transferring to our $Z_n$, satisfies
\[
    \mathbb{V}(Z_n) = \binom{n+1}2 -\mu_n^2 -
    \sum_{0\le k\le n-2} \frac{(-2)^k}{(k+2)!}
    \binom{n-k}{2}\left(2^k(k-2)+k^2+4k+6\right);
\]
asymptotically,
\begin{align}\label{VZn}
    \mathbb{V}(Z_n) = e^{-4}(n+3) + O\left(
    \frac{4^n}{(n+2)!}\right).
\end{align}
Another interesting result in \cite{Page59} is the closed-form
expression for the bivariate generating function of $\mathbb{E}
(t^{\zeta_n})$, obtained by solving a Riccati equation; see also
\cite{Runnenburg82}. In terms of $Z_n$, this closed-form translates
into
\begin{align} \label{one-row-Fzt}
    \sum_{n\ge0} \mathbb{E}\left(t^{Z_n}\right) z^n
    = \frac{\sqrt{t}\left(\left( 1+\sqrt{t}\right)
    e^{2\sqrt{t}z}+1-\sqrt{t}\right)}{\left(1+ \sqrt{t}\right)
    \left(1-\sqrt{t}z\right) e^{2\sqrt{t}z}
    -\left( 1-\sqrt{t}\right) \left( 1+\sqrt{t}z\right) }.
\end{align}
Page predicted that the $\zeta_n$'s were asymptotically normally
distributed, which was later proved by Runnenburg
\cite{Runnenburg82} by the method of moments; see \cite{KR03} for an
extension. See also \cite{Downton61,DSI11,Flajolet97} for other
properties studied. The asymptotic normality is contained as a
special case of Penrose and Sudbury's very general central limit
theorem in \cite{PS05}, where they also derived a convergence rate
by Stein's method.

The exact solvability of such a model is however very rare in the
literature, and the next possibly solvable cases are the unfriendly
variants for two rows of seats with the same rule of nearest
neighbors exclusion, which we may refer to as the \emph{unfriendly
seating arrangement in a dining table}. Such a model and the like
were studied by physicists in the 1990's and the ``jamming density''
(the large-$n$ limit of the ratio between the expected number of
persons sitting down and the total number of seats) was given
explicitly by
\begin{align}\label{2R-limit}
    \tfrac{1}{4}\left( 2-e^{-1}\right)\approx 0.408030\dots
\end{align}
using different heuristic arguments; see \cite{BK92,FP92}. This
constant is to be compared with that in the one-row case
\[
    \tfrac{1}{2}\left( 1-e^{-2}\right)\approx 0.432332\dots;
\]
see also Finch's book \cite[\S 5.3.1]{Finch03} for more information.

The same problem was recently reformulated as the \emph{unfriendly
theater seating arrangement problem} by Georgiou et al.
\cite{GKK09}, where they indeed addressed the configuration of $m$
rows of seats (mentioned to be connected to maximal independent sets
of planar lattice) and proved the existence of the expected
proportion of occupied seats. In particular, they also derived the
jamming limit \eqref{2R-limit}. Unfortunately, the crucial
stochastic dominance relations used in their paper \cite{GKK09} are
incorrect, and thus their proofs remain incomplete (the asymptotic
linearity being well expected though). More precisely, they claimed
that if $H$ is an induced subgraph of $G$, then the
\emph{first-order stochastic dominance} relation $X_H\le X_G$ holds
in the sense that
\begin{equation} \label{g1}
    \mathbb{P}\left( X_{H}> k\right)
    \le \mathbb{P}\left( X_{G}>k\right) ,
\end{equation}
for all $k$, where $X_{H},X_{G}$ are the random variables counting
the number of occupied seats (or the cardinality of an independent
set) when starting with the seat configurations $H$ and $G$,
respectively, and following the same random unfriendly seating
procedure until the procedure terminates. It is known that this
implies the \emph{second-order stochastic dominance}
\begin{equation} \label{g2}
    \mathbb{E}(X_H)\le \mathbb{E}(X_G).
\end{equation}

Unfortunately, none of the two relations \eqref{g1} and \eqref{g2}
is correct. Here is a counterexample to \eqref{g1}. If the two
initial seat configurations are given as follows
\[
    H_1=
    \begin{array}{lll}
    & \OE &  \\
    \OE & \OE & \OE
    \end{array}
    \quad \text{and}\quad G_1=
    \begin{array}{lll}
    \OE & \OE &  \\
    \OE & \OE & \OE
    \end{array},
\]
then
\[
    \begin{cases}
        \mathbb{P}(X_{H_1}=1) = \frac14\\
        \mathbb{P}(X_{H_1}=3) = \frac34,
    \end{cases}
    \quad \text{and}\quad
    \begin{cases}
        \mathbb{P}(X_{G_1}=2) = \frac7{15}\\
        \mathbb{P}(X_{G_1}=3) = \frac8{15}.
    \end{cases}
\]
implying that
\[
    \mathbb{P}(X_{H_1}\ge 3)>\mathbb{P}(X_{G_1}\ge 3),
\]
contrary to (\ref{g1}). For a counterexample to \eqref{g2}, consider
the following two seat configurations
\[
    H_{2}=  \begin{array}{ll}
    \bigcirc &  \\
    & \bigcirc
    \end{array},\quad
    G_{2}=  \begin{array}{ll}
    \bigcirc &  \\
    \bigcirc & \bigcirc
    \end{array}.
\]
Then the expected numbers of occupied seats satisfy
$\mathbb{E}[X_{H_2}]=2>\mathbb{E}[X_{G_2}]=5/3$.

Due to the subtlety of the problem, we focus our attention in this
paper on the dining table model and we show that this model is also
explicitly solvable by solving a system of nonlinear differential
equations. This new result leads to interesting structural properties,
and many strong limit theorems will then follow. In particular, our
analysis provides the first rigorous, complete proof for the very
simple jamming limit \eqref{2R-limit} with an optimal error terms.
Some related stochastic dominance relations will be clarified in
Section~\ref{sec:sd}.

\section{Recurrences and solutions}

We consider a dining table with $2n$ seats arranged in two rows
\begin{align}\label{XXn}
    \mathscr{X}_n :=
    \begin{array}{l}
    \overbrace{
    \begin{array}{lllllllll}
        \OE & \OE & \OE & \OE
        & \OE & \cdots & \OE & \OE & \OE
    \end{array}}^{n}\\
    \begin{array}{lllllllll}
        \OE & \OE & \OE & \OE
        & \OE  & \cdots & \OE & \OE & \OE
    \end{array}
    \end{array}
\end{align}
Diners arrive one after another and each selects a seat uniformly at
random. If the seat is empty, then it becomes occupied, and two of
its neighboring seats together with the opposite one (in the other
row) are no more available. If the seat selected is occupied or
forbidden and there are still empty seats available, then the
(uniform) random selection is repeated until a seat is found. The
process stops as long as all seats are either occupied or forbidden.
An example with $n=10$ is given as follows (where ``$\ON$" stands
for a forbidden seat and ``$\OY$'' an occupied seat)
\[
    \begin{array}{llllllllll}
        \ON & \OY & \ON & \ON
        & \ON & \OY & \ON & \ON
        & \OY & \ON \\
        \OY & \ON & \ON & \OY
        & \ON & \ON & \OY & \ON
        & \ON & \OY
    \end{array}.
\]

Let $X_n$ count the total number of persons sitting down when such a
sequential process terminates. Then it is easy to see that
\[
    \tr{n/2}+1\le X_n \le n \qquad(n\ge1).
\]
By splitting the $2n$-problem at the first occupied seat into two
subproblems, we are then led to the recurrence relation for the
probability generating function $X_n(t) := \mathbb{E}(t^{X_n})$
\begin{align}\label{Xnt-rr}
    X_n(t) = \frac{t}{n}\sum_{0\le k\le n-1}Y_k(t) Y_{n-1-k}(t)
    \qquad(n\ge1),
\end{align}
with $X_0(t) = 1$. Here $Y_n$ counts the number of occupied seats
under the same unfriendly seating procedure but with the slightly
different initial configuration of the seats
\begin{align}\label{YYn}
    \mathscr{Y}_n :=
    \begin{array}{l}
    \overbrace{
    \begin{array}{llllllll}
        \OE & \OE & \OE & \OE
        & \OE & \cdots & \OE & \OE
    \end{array}}^{n-1}\\
    \begin{array}{lllllllll}
        \OE & \OE & \OE & \OE
        & \OE & \cdots & \OE & \OE & \OE
    \end{array}
    \end{array}
\end{align}
where the total number of seats is $2n-1$. The following two
diagrams show the obvious decompositions after the first seat is
occupied.
\begin{center}
\begin{tikzpicture}[scale=0.8, every node/.style={scale=0.8}]
\node[] at (0,0) {
\begin{tikzpicture}[scale=1.25,transform shape,
cross/.style={path picture={ \draw[black]
(path picture bounding box.south east)
-- (path picture bounding box.north west)
(path picture bounding box.south west)
-- (path picture bounding box.north east);
}}
]
\draw[thick] (0,0) -- (1,0) -- (1,0.5) -- (1.5,0.5)
-- (1.5,1) -- (0,1) -- (0,0);
\draw[thick] (2.5,0) -- (4.5,0) -- (4.5,1) -- (2,1)
-- (2,0.5) -- (2.5,0.5) -- (2.5,0);
\node [] at (2.25,0.25){\scalebox{1.5}{$\oslash$}};
\node [] at (1.25,0.25){\scalebox{1.5}{$\oslash$}};
\node [circle,fill,scale=1.05] at (1.75,0.25){};
\node [] at (1.75,0.75){\scalebox{1.5}{$\oslash$}};
\node [] at (0.5,0.5){$Y_\bullet$};
\node [] at (3.5,0.5){$Y_\bullet$};
\end{tikzpicture}
};

\node[] at (7,0) {
\begin{tikzpicture}[scale=1.25,transform shape,
cross/.style={path picture={ \draw[black]
(path picture bounding box.south east)
-- (path picture bounding box.north west)
(path picture bounding box.south west)
-- (path picture bounding box.north east);
}}
]
\draw[thick] (0,0) -- (1.5,0) -- (1.5,0.5) -- (1,0.5)
-- (1,1) -- (0,1) -- (0,0);
\draw[thick] (2,0) -- (4.5,0) -- (4.5,1) --(2.5,1)
--(2.5,0.5) --(2,0.5) -- (2, 0);
\node [circle,fill,scale=1.05] at (1.75,0.75){};
\node [] at (1.75,0.25){\scalebox{1.5}{$\oslash$}};
\node [] at (2.25,0.75){\scalebox{1.5}{$\oslash$}};
\node [] at (1.25,0.75){\scalebox{1.5}{$\oslash$}};
\node [] at (0.5,0.5){$Y_\bullet$};
\node [] at (3.5,0.5){$Y_\bullet$};
\end{tikzpicture}
};
\end{tikzpicture}
\end{center}

Applying the same conditioning argument to $Y_n$, we need to
introduce two additional sequences of random variables based on the
following seat configurations: for $n\ge1$
\begin{align*}
    \mathscr{A}_{-1}&=\mathscr{A}_{0}=\varnothing,\\
    \mathscr{A}_n &:=
    \begin{array}{l}
    \overbrace{
    \begin{array}{lllllllll}
        \OE & \OE & \OE & \OE
        & \OE & \cdots & \OE & \OE &\OE
    \end{array}}^{n}\\
    \hspace*{.82cm}
    \begin{array}{lllllllll}
        \OE & \OE & \OE & \OE
        & \OE & \cdots & \OE & \OE & \OE
    \end{array}
    \end{array}
\end{align*}
and
\begin{align*}
    \mathscr{B}_{-1}&=\varnothing, \;
    \mathscr{B}_{0}=\OE,\\
    \mathscr{B}_n &:= \hspace*{-0.9cm}
    \begin{array}{c} \hspace*{.82cm}
    \overbrace{
    \begin{array}{llllllll}
        \OE & \OE & \OE
        & \OE & \cdots & \OE & \OE &\OE
    \end{array}}^{n-1}\\
    \hspace*{.82cm}
    \begin{array}{llllllllll}
        \OE & \OE & \OE & \OE
        & \OE &  \cdots & \OE & \OE & \OE &\OE
    \end{array}
    \end{array}
\end{align*}
Let $A_{n}$, and $B_{n}$ denote the number of sitting persons under
the same unfriendly seating procedure when started from the
configurations $\mathscr{A}_n$ and $\mathscr{B}_n$, respectively.
The initial conditions are defined to be $A_{-1}=A_0=0$ and
$B_{-1}=0,B_{0}=1$. Then we have the following systems of
recurrences.
\begin{lmm} The probability generating functions
$A_n(t), B_n(t)$ and $Y_n(t)$ satisfy \label{lmm-rr}
\begin{align} \label{m22}
    \begin{cases}\displaystyle
    A_n(t) = \frac{t}{n}\sum_{1\le k\le n}A_{k-2}(t) B_{n-k}(t),\\
    \displaystyle
    B_n(t) = \frac{t}{2n} \left(\sum_{1\le k\le n-1}
    B_{k-1}(t)B_{n-1-k}(t) + \sum_{1\le k\le n+1} A_{k-2}(t)
    A_{n-k}(t)\right),
    \end{cases}
\end{align}
and
\[
    Y_n(t) = \frac{t}{2n-1} \left(\sum_{0\le k\le n-2}
    Y_k(t)B_{n-2-k}(t) + \sum_{0\le k\le n-1} Y_k(t)
    A_{n-2-k}(t)\right),
\]
for $n\ge1$ with the initial conditions $A_n(t)=B_n(t)=Y_n(t)=1$ if
$n<0$ and $A_0(t)=Y_0(t)=1$ and $B_0(t)=t$.
\end{lmm}
\begin{proof}
After the first diner sits down, the random variable $Y_n$ is
decomposed in either the following two ways.
\begin{center}
\begin{tikzpicture}[scale=0.8, every node/.style={scale=0.8}]
\node[] at (0,0) {
\begin{tikzpicture}[scale=1.25,transform shape,
cross/.style={path picture={ \draw[black]
(path picture bounding box.south east)
-- (path picture bounding box.north west)
(path picture bounding box.south west)
-- (path picture bounding box.north east);
}}
]
\draw[thick] (0,0) -- (1.5,0) -- (1.5,0.5) --
(1,0.5) -- (1,1) -- (0,1) -- (0,0);
\draw[thick] (2,0) -- (5,0) --(5,0.5)--
(4.5,0.5)--(4.5,1) --(2.5,1) --(2.5,0.5) --(2,0.5) -- (2, 0);
\node [circle,fill,scale=1.05] at (1.75,0.75){};
\node [] at (1.75,0.25){\scalebox{1.5}{$\oslash$}};
\node [] at (2.25,0.75){\scalebox{1.5}{$\oslash$}};
\node [] at (1.25,0.75){\scalebox{1.5}{$\oslash$}};
\node [] at (0.5,0.5){$Y_\bullet$};
\node [] at (3.5,0.5){$B_\bullet$};
\end{tikzpicture}
};
\node[] at (7,0) {
\begin{tikzpicture}[scale=1.25,transform shape,
cross/.style={path picture={ \draw[black]
(path picture bounding box.south east)
-- (path picture bounding box.north west)
(path picture bounding box.south west)
-- (path picture bounding box.north east);
}}
]
\draw[thick] (-1,0) -- (1,0) -- (1,0.5) --
(1.5,0.5) -- (1.5,1) -- (-1,1) -- (-1,0);
\draw[thick] (2.5,0) -- (4,0)--(4,0.5)--(3.5,0.5) --
(3.5,1) -- (2,1) -- (2,0.5) -- (2.5,0.5) -- (2.5,0);
\node [] at (2.25,0.25){\scalebox{1.5}{$\oslash$}};
\node [] at (1.25,0.25){\scalebox{1.5}{$\oslash$}};
\node [circle,fill,scale=1.05] at (1.75,0.25){};
\node [] at (1.75,0.75){\scalebox{1.5}{$\oslash$}};
\node [] at (0,0.5){$Y_\bullet$};
\node [] at (3,0.5){$A_\bullet$};
\end{tikzpicture}
};
\end{tikzpicture}
\end{center}
Similarly, the random variable $A_n$ is decomposed as follows
\begin{center}
\begin{tikzpicture}[scale=0.8, every node/.style={scale=0.8}]
\node[] at (0,0) {
\begin{tikzpicture}[scale=1.25,transform shape,
cross/.style={path picture={ \draw[black]
(path picture bounding box.south east)
-- (path picture bounding box.north west)
(path picture bounding box.south west)
-- (path picture bounding box.north east);
}}
]
\draw[thick] (-0.5,0) -- (1.5,0) -- (1.5,0.5) -- (1,0.5) -- (1,1)
-- (-1,1)--(-1,0.5)--(-0.5,0.5) --(-0.5,0);
\draw[thick] (2,0) -- (4,0) --(4,0.5)-- (3.5,0.5)
--(3.5,1) --(2.5,1) --(2.5,0.5) --(2,0.5) -- (2, 0);
\node [circle,fill,scale=1.05] at (1.75,0.75){};
\node [] at (1.75,0.25){\scalebox{1.5}{$\oslash$}};
\node [] at (2.25,0.75){\scalebox{1.5}{$\oslash$}};
\node [] at (1.25,0.75){\scalebox{1.5}{$\oslash$}};
\node [] at (0.3,0.5){$A_\bullet$};
\node [] at (3,0.5){$B_\bullet$};
\end{tikzpicture}
};
\node[] at (7,0) {
\begin{tikzpicture}[scale=1.25,transform shape,
cross/.style={path picture={ \draw[black]
(path picture bounding box.south east)
-- (path picture bounding box.north west)
(path picture bounding box.south west)
-- (path picture bounding box.north east);
}}
]
\draw[thick] (0,0) -- (1,0) -- (1,0.5) -- (1.5,0.5) -- (1.5,1)
-- (-0.5,1)--(-0.5,0.5)--(0,0.5) -- (0,0);
\draw[thick] (2.5,0) -- (5,0)--(5,0.5)--(4.5,0.5) -- (4.5,1)
-- (2,1) -- (2,0.5) -- (2.5,0.5) -- (2.5,0);
\node [] at (2.25,0.25){\scalebox{1.5}{$\oslash$}};
\node [] at (1.25,0.25){\scalebox{1.5}{$\oslash$}};
\node [circle,fill,scale=1.05] at (1.75,0.25){};
\node [] at (1.75,0.75){\scalebox{1.5}{$\oslash$}};
\node [] at (0.5,0.5){$B_\bullet$};
\node [] at (3.5,0.5){$A_\bullet$};
\end{tikzpicture}
};
\end{tikzpicture}
\end{center}
And, finally, we have the two possible decompositions for $B_n$
\begin{center}
\begin{tikzpicture}[scale=0.8, every node/.style={scale=0.8}]
\node[] at (0,0) {
\begin{tikzpicture}[scale=1.25,transform shape,
cross/.style={path picture={ \draw[black]
(path picture bounding box.south east)
-- (path picture bounding box.north west)
(path picture bounding box.south west)
-- (path picture bounding box.north east);
}}
]
\draw[thick] (0,0) -- (1.5,0) -- (1.5,0.5) -- (1,0.5)
-- (1,1) -- (0,1)-- (0,0.5)--(-0.5,0.5) --(-0.5,0)-- (0,0);
\draw[thick] (2,0) -- (5,0) --(5,0.5)-- (4.5,0.5)--(4.5,1)
--(2.5,1) --(2.5,0.5) --(2,0.5) -- (2, 0);
\node [circle,fill,scale=1.05] at (1.75,0.75){};
\node [] at (1.75,0.25){\scalebox{1.5}{$\oslash$}};
\node [] at (2.25,0.75){\scalebox{1.5}{$\oslash$}};
\node [] at (1.25,0.75){\scalebox{1.5}{$\oslash$}};
\node [] at (0.5,0.5){$B_\bullet$};
\node [] at (3.5,0.5){$B_\bullet$};
\end{tikzpicture}
};
\node[] at (7,0) {
\begin{tikzpicture}[scale=1.25,transform shape,
cross/.style={path picture={ \draw[black]
(path picture bounding box.south east)
-- (path picture bounding box.north west)
(path picture bounding box.south west)
-- (path picture bounding box.north east);
}}
]
\draw[thick] (-1,0) -- (1,0) -- (1,0.5) -- (1.5,0.5)
-- (1.5,1) --(-0.5,1)--(-0.5,0.5) -- (-1,0.5)--(-1,0);
\draw[thick] (2.5,0) -- (4,0)--(4,0.5)--(3.5,0.5)
-- (3.5,1) -- (2,1) -- (2,0.5) -- (2.5,0.5) -- (2.5,0);
\node [] at (2.25,0.25){\scalebox{1.5}{$\oslash$}};
\node [] at (1.25,0.25){\scalebox{1.5}{$\oslash$}};
\node [circle,fill,scale=1.05] at (1.75,0.25){};
\node [] at (1.75,0.75){\scalebox{1.5}{$\oslash$}};
\node [] at (0.2,0.5){$A_\bullet$};
\node [] at (3,0.5){$A_\bullet$};
\end{tikzpicture}
};
\end{tikzpicture}
\end{center}
The lemma follows by computing the corresponding probabilities.
\end{proof}

Consider now $G_A(z,t) := \sum_{n\ge0}\mathbb{E}
\left(t^{A_n}\right)z^n$, the bivariate generating function of
$A_n$. The notations $G_B(z,t)$ and $G_Y(z,t)$ are defined
similarly. Then Lemma~\ref{lmm-rr} implies the following system of
Riccati equations.
\begin{lmm} The bivariate generating functions $G_A, G_B$ satisfy
\begin{align*}
    \begin{cases}
        G_A' =tG_B+tzG_AG_B, \\ \displaystyle
        G_B' =tG_A+\frac{tz}{2}\left(
        G_A^{2}+G_B^{2}\right),
    \end{cases}
\end{align*}
with $G_A(0,t)=1$ and $G_B(0,t)=t$, and
\[
    2zG_Y'= \left(1+tz+tz^2(G_A+G_B)\right) G_Y-1,
\]
with $G_Y(0,t) = 1$. Here for simplicity $G_\bullet =
G_\bullet(z,t)$ and $G_\bullet' := (\partial/\partial
z)G_\bullet(z,t)$.
\end{lmm}
These equations admit explicit solutions as follows. Define
\begin{align}\label{Uzt}
    U(z,t)=\frac{2t(1+t)}{(1+t)(1-tz)-(1-t)e^{-tz}}.
\end{align}
\begin{lmm} We have
\begin{align}\label{FAFB}
    G_A(z,t)=\frac{U(z,t)+U(z,-t)}2 ,
    \quad G_B(z,t)=\frac{U(z,t)-U(z,-t)}2,
\end{align}
and
\begin{align}\label{FY-PQ}
    G_Y(z,t) = \frac{Q(z,t)}{P(z,t)},
\end{align}
where
\begin{align*}
    P(z,t) &= (1+t)(1-tz)-(1-t)e^{-tz},\\
    Q(z,t) &= 1+t-(1-t)e^{-tz}-
    \frac{1-t}2\,tz\int_0^1 e^{-tz(1+v)/2}v^{-1/2}
    \dd v.
\end{align*}
\end{lmm}
Note that $Q$ can be expressed in terms of the error function or the
standard normal distribution function $\Phi$. For example,
\[
    Q(z,t) = 1+t-(1-t)e^{-tz/2}
    \left(e^{-tz/2} - \sqrt{\frac{\pi}{2} tz}
    +\sqrt{2\pi tz}\,\Phi(\sqrt{tz})\right).
\]
\begin{proof}
For convenience, define $V(z,t) := U(z,-t)$. Then $U=G_A+G_B$ and
$V=G_A-G_B$ satisfy the simpler equations
\begin{align*}
    \begin{cases} \displaystyle
    U'=tU+\frac{tz}{2}U^2, \\ \displaystyle
    V'=-tV-\frac{tz}{2}V^2,
    \end{cases}
\end{align*}
with $U(0,t)=1+t$ and $V(0,t)=1-t$. Since this is a system of
Bernoulli equations, we consider the transformation $u=-U^{-1}$,
which satisfies the equation
\[
    u'+tu=\frac{tz}{2},
\]
with $u(0,t)=-1/(1+t)$. Solving this equation gives
\eqref{Uzt}, and \eqref{FAFB} follows.

For $G_Y$, we then have the first-order differential equation
\[
    2zG_Y'= \left(1+tz+tz^2U\right) G_Y-1.
\]
To solve this equation, we consider $\tilde{G}_Y := G_Y-1$ and
introduce the integration factor
\[
    I(z) = I(z,t) := z^{-1/2} e^{tz/2} P(z,t).
\]
Then
\[
    (I\cdot \tilde{G}_Y)' = \frac{tI}{2} (1+zU),
\]
with $\tilde{G}_Y(0,t)=0$, which has the solution
\[
   \tilde{G}_Y(z,t)
    = \frac{t}{2I(z)}\int_0^z I(s) (1+sU(s,t))\dd s.
\]
Note that
\begin{align*}
    I(z,t)(1+zU(z,t)) = z^{-1/2} \left( (1+t)(1+zt)
    e^{tz/2}-(1-t)e^{-tz/2}\right) .
\end{align*}
Then
\begin{align*}
    &\int_0^z I(s,t) \left(1+s U(s,t)\right)\dd{s} \\
    &\qquad= 2\sqrt{z}(1+t)e^{tz/2}-(1-t)\int_0^z s^{-1/2}
    e^{-ts/2}\dd{s}\\
    &\qquad=2\sqrt{z}(1+t){\rm e}^{tz/2}-(1-t)
    \sqrt{z}\int_0^1 v^{-1/2}e^{-tzv/2}\dd{v}.
\end{align*}
Thus
\begin{align*}
    \widetilde{G}_Y(z,t)
    &=\frac{Q(z,t)}{P(z,t)}
    =1+\frac {t}{2I(z,t)} \int_0^z I(s,t)
    \left(1+s U(s,t)\right)\dd{s}\\
    &= \frac1{P(z,t)}\left(1+t-(1-t)
    e^{-tz}-\frac{t(1-t)}2 z \int_0^1
    e^{-tz(1+v)/2}v^{-1/2}\dd{v}\right),
\end{align*}
which proves \eqref{FY-PQ}.
\end{proof}

Returning to $X_n$, by \eqref{Xnt-rr}, we have
\begin{align}\label{FX}
    G_X(z,t) := \sum_{n\ge0}\mathbb{E}
    \left(t^{A_n}\right)z^n
    = 1+t\int_0^zG_Y(u,t)^2 \dd u.
\end{align}

Since the uniform splitting procedure also arises naturally in
diverse algorithmic and combinatorial contexts, Riccati equations
were often encountered in related literature; see, for example,
\cite{FGM97,PP98}.

\section{Mean and variance}
\label{sec:mv}

With the explicit expressions derived above, we have two different
approaches to compute the mean and the variance: one based on a
direct use of \eqref{FX} and a suitable manipulation of the error
terms (see \cite[Ch.\ VII]{FS09}) and the other depending on
Quasi-Power type argument (see \cite[\S IX.
5]{FS09},\cite{Hwang98}). While both approaches provide readily the
two dominant asymptotic terms, the characterization of the extremely
small error requires a more careful analysis. For methodological
interest, we discuss the first approach here by providing a general
means for error analysis, which will also be useful for problems of
a similar nature. The second approach will be briefly indicated
later.

\begin{thm} The mean of $X_n$ satisfies
\begin{align}\label{EXn2}
    \mathbb{E}(X_n) = \mu n +c_1 + O\left(\frac1{(n+3)!}\right),
\end{align}
where $\mu := 1-e^{-1}/2$ and ($\phi := 2\Phi(1)-1$)
\begin{align}\label{c1}
    c_1 := \frac{e^{-1}}{2}\left(\sqrt{2\pi e}\,\phi-1\right)
    \approx 0.33502\,27062\,94844\dots,
\end{align}
and the variance satisfies
\begin{align}\label{VXn}
    \mathbb{V}(X_n) = \sigma^2 n + c_2 +
    O\left(\frac{2^n}{(n+4)!}\right),
\end{align}
where $\sigma := \sqrt{\frac34}\,e^{-1}$ and
\begin{align}\label{c2}
\begin{split}
    c_2 := \frac{e^{-2}}{4}\left(-\pi e\phi^2
    -2\sqrt{2\pi e}\,\phi +5\right)
    \approx -0.15640\,75038\,00915\dots
\end{split}
\end{align}
\end{thm}
From \eqref{EXn2}, we see that the jamming density is given by
\[
    \lim_{n\to\infty}\frac{\mathbb{E}(X_n)}{2n}
    = \tfrac14\left(2-e^{-1}\right).
\]
Also the $O$-terms in \eqref{EXn2} and \eqref{VXn} are smaller than
the corresponding ones in the one-row version.
\begin{proof}
From \eqref{FY-PQ}, we have
\begin{align*}
    M_Y(z) &:= \sum_{n\ge0}\mathbb{E}(Y_n) z^n
    =\frac{\partial}{\partial t}\, G_Y(z,t) \Bigr|_{t=1}\\
    &= \frac{z}{(1-z)^2}\left(1-\frac{e^{-z}}{2}\right)
    +\frac{z}{4(1-z)}\int_0^1 v^{-1/2}e^{-(1+v)z/2}\dd v.
\end{align*}
Then we deduce that
\begin{align}
    M_X(z) &:= \sum_{n\ge0}\mathbb{E}(X_n) z^n
    = 2\int_0^z\frac{M_Y(u)}{1-u}\dd u +
    \int_0^z \frac1{(1-u)^2} \dd u \label{MXz}\\
    &= \frac{\mu}{(1-z)^2}
    + \frac{c_0}{1-z} + O(1), \nonumber
\end{align}
as $z\sim 1$, where
\[
    c_0 := -1 + \frac{\sqrt{\pi}}{\sqrt{2e}}
    \left(2\Phi(1)-1\right)=-1 + \frac{\sqrt{\pi}}{\sqrt{2e}}
    \,\phi.
\]
Consequently, by standard singularity analysis \cite[Ch.\ VII]{FS09},
\begin{align}\label{EXn}
    \mathbb{E}(X_n) = \mu n +c_1 + O\left(n^{-K}\right),
\end{align}
for any $K>0$.

The leading terms in \eqref{VXn} for the variance are computed
similarly.

Numerically, the approximation \eqref{EXn} without the $O$-term is
extremely good even for small values of $n$; see Figure~\ref{fg-EV}.
For example, the error term is already less than $10^{-7}$ when
$n\ge8$.

\begin{figure}[!h]
\begin{center}
\includegraphics[width=5.5cm]{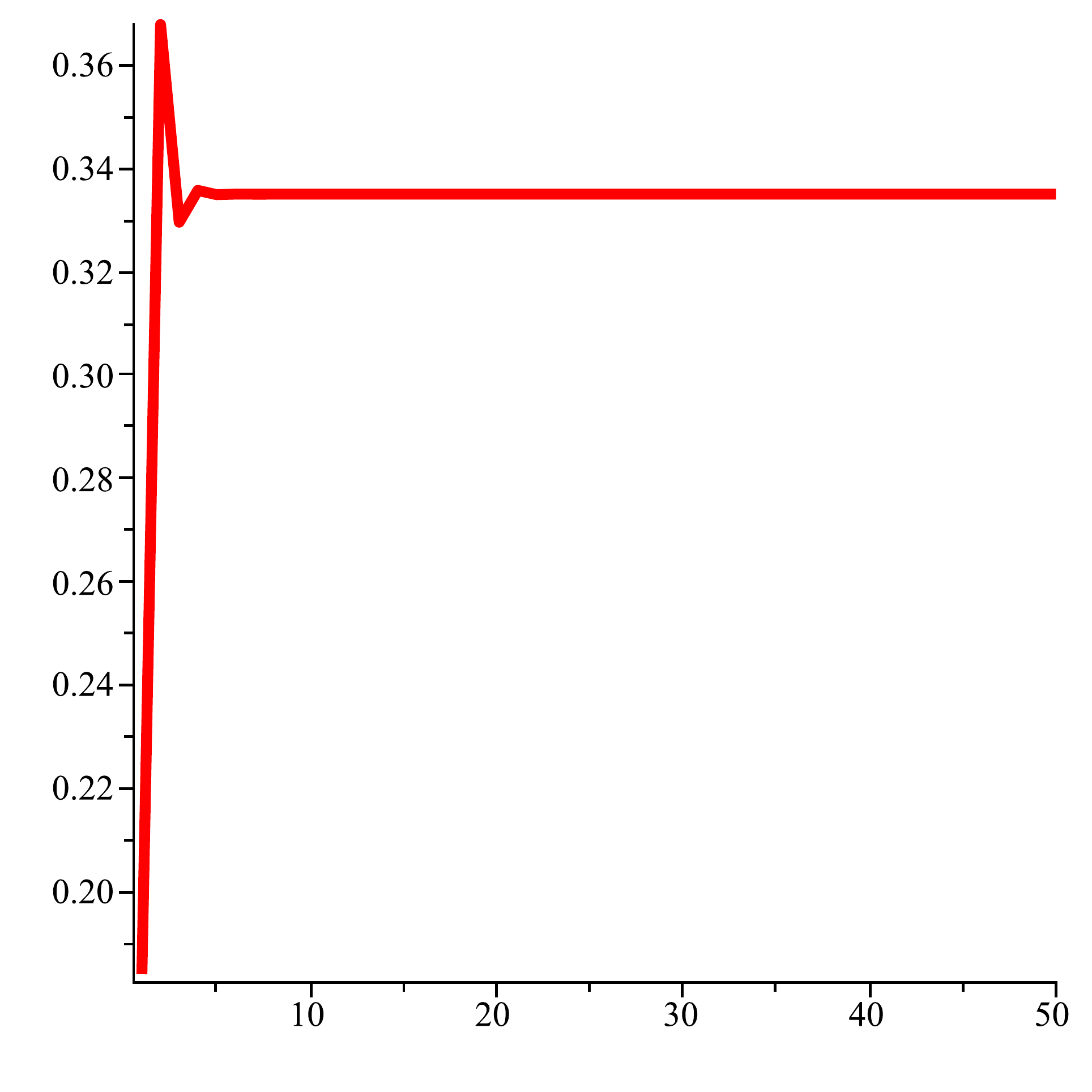}\;\;
\includegraphics[width=5.5cm]{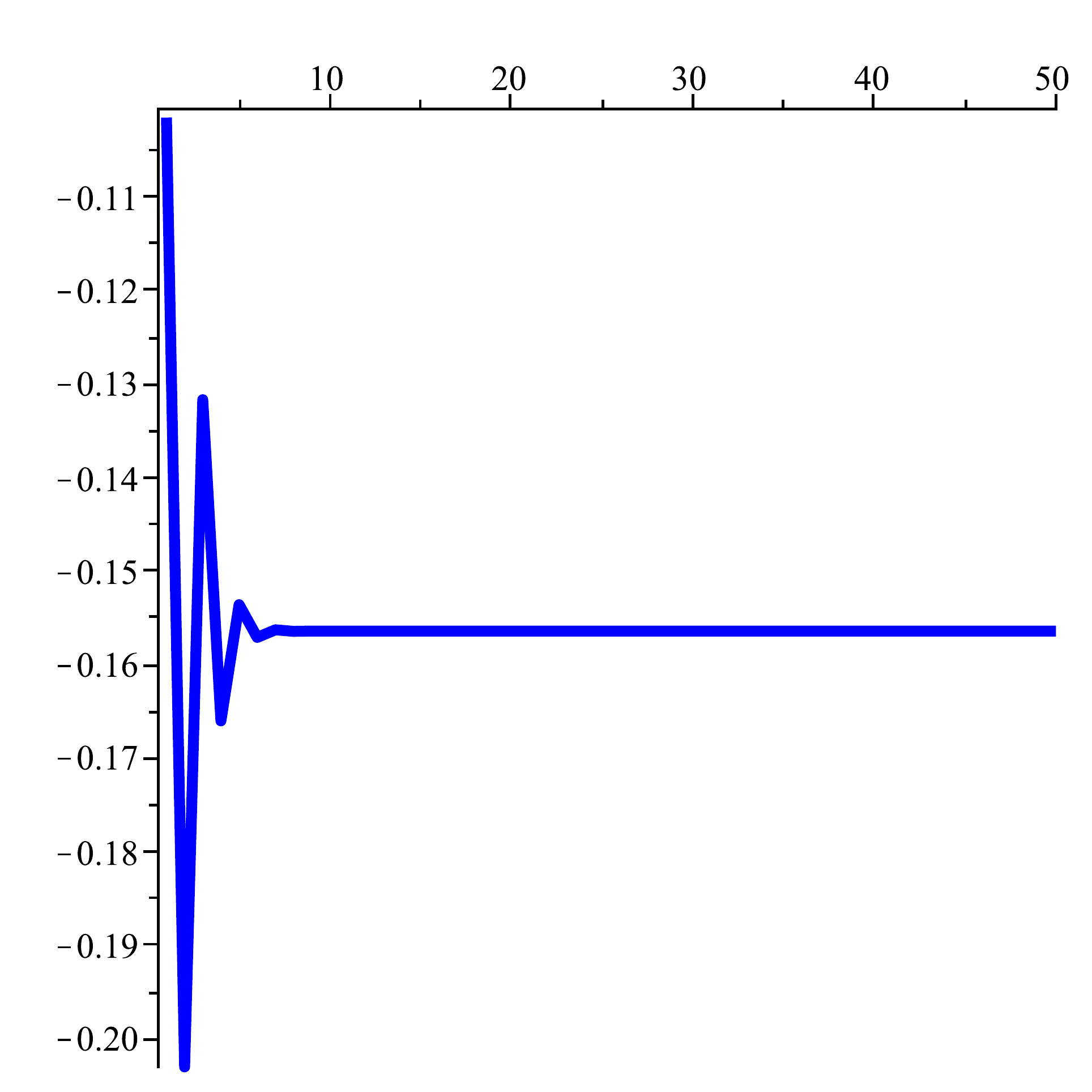}
\end{center}
\vspace*{-.5cm}
\caption{\emph{Goodness of the approximations \eqref{EXn}
and \eqref{VXn} by computing the exact values of
$\mathbb{E}(X_n)-\mu n$ (left) and $\mathbb{V}(X_n)-
\sigma^2 n$ (right) for $n=1,\dots,50$.}}
\label{fg-EV}
\end{figure}

To clarify the rapid convergence of the mean and variance towards
their limit (see Figure~\ref{fg-EV}), we refine the asymptotic
approximation \eqref{EXn} by the following simple error analysis.
Note first that we are dealing with asymptotics of the form
($[z^n]f(z)$ denoting the coefficient of $z^n$ in the Taylor
expansion of $f$)
\[
    [z^n] \frac{f(z)}{(1-z)^m},
    \quad [z^n] \int_0^z \frac{f(u)}{(1-u)^m} \,\dd u,
\]
where $m=1,2,\dots$ and $f$ is an entire function with quickly
decreasing coefficients.

\begin{prop} If $f$ is an entire function whose coefficients
satisfy \label{prop-facto}
\[
    [z^n]f(z) = O\left(\varepsilon_n\right),
\]
where $\ve_n$ is a \emph{positive} sequence satisfying
$\varepsilon_n= O(K^{-n})$ for some $K>1$, then, for $m=1,2,\dots$,
\begin{align}\label{fQz}
    [z^n]\frac{f(z)}{(1-z)^m} =
    \sum_{0\le j<m} \frac{(-1)^j}{j!}\, f^{(j)}(1)
    \binom{n+m-1-j}{m-1-j} + O(|\varepsilon_{n+m}|),
\end{align}
and
\begin{align}\label{fQz2}
    [z^n]\int_0^z\frac{f(u)}{(1-u)^m}\,\dd u =
    \frac1n\sum_{0\le j<m} \frac{(-1)^j}{j!}\, f^{(j)}(1)
    \binom{n+m-2-j}{m-1-j} + O\left(
    \frac{|\varepsilon_{n+m-1}|}{n}\right).
\end{align}
\end{prop}
\begin{proof}
Let $f_n := [z^n]f(z)$. Then
\begin{align*}
    [z^n]\frac{f(z)}{(1-z)^m}
    = \sum_{0\le k\le n} \binom{n+m-k-1}{m-1} f_k
    = \sum_{k\ge0}\binom{n+m-k-1}{m-1}  f_k
    - \delta_n,
\end{align*}
where
\begin{align*}
    \delta_n := \sum_{k\ge n+m}
    \binom{n+m-k-1}{m-1}f_k
    = O\left(|\varepsilon_{n+m}|)\right).
\end{align*}
On the other hand, by expanding $f(z)$ at $z=1$ and computing the
coefficients term by term, we have the identity
($f_k=f^{(k)}(0)/k!$)
\[
    \sum_{k\ge0}\frac{f^{(k)}(0)}{k!}\binom{n+m-k-1}{m-1}
    = \sum_{0\le j<m} (-1)^j\frac{f^{(j)}(1)}{j!}
    \binom{n+m-1-j}{m-1-j}.
\]
This proves \eqref{fQz}. For \eqref{fQz2}, we have
\[
    [z^n]\int_0^z\frac{f(u)}{(1-u)^m}\,\dd u
    = \frac1{n}\sum_{0\le k<n} \binom{n+m-k-2}{m-1} f_k .
\]
Then \eqref{fQz2} follows by the same analysis by replacing $n$ by
$n-1$.
\end{proof}
Our analysis indeed applies to a wider class of $f$ but we do not
need this in this paper.

We now apply this lemma to $M_X(z)$, which has the form
\[
    M_X(z) = \int_0^z \frac{f_1(u)}{(1-u)^3} \dd u,
\]
where
\begin{align*}
    f_1(z) = 1+z- \frac{z^2}2\int_0^1
    v^{-1/2} (1-v) e^{-(1+v)z/2} \dd v.
\end{align*}
Thus for $n\ge2$
\[
    [z^n]f_1(z) = \frac{(-1)^{n-1}}{(n-2)!}
    \sum_{0\le j\le n-2} \binom{n-2}{j}(-1)^j
    \frac{(j+1)!\sqrt{\pi}}{\Gamma(j+5/2)2^{j+1}}.
\]
By the standard integral representation for finite
differences (Rice's formula), we deduce that
\begin{align}\label{f1z-asymp}
    [z^n]f_1(z) = 2\frac{(-1)^{n-1}}{n!}
    \left(1+\frac{2}{n+1}+O\left(n^{-2}\right)\right).
\end{align}
Indeed, one obtains the (divergent) full asymptotic expansion
\[
    [z^n]f_1(z) \sim 2\frac{(-1)^{n-1}}{n!}
    \left(1+\sum_{k\ge3}
    \frac{(k-1)(2k-4)!}{2^{k-2}(k-2)!}
    \cdot\frac1{(n+1)\cdots(n+k-2)}
    \right).
\]
On the other hand, we also have
\[
    f_1(1) = 2-e^{-1},\quad
    \text{and} \quad f_1'(1) = c_0;
\]
see \eqref{MXz}. Applying now \eqref{fQz2} gives not only
the leading terms $\mu n+ c_1$ for $\mathbb{E}(X_n)$ but
also the precise error term in \eqref{EXn2}.

In a similar way, we have
\begin{align*}
    S_Y(z) &:= \sum_{n\ge0}\mathbb{E}(Y_n^2)z^n
    = \frac{\partial^2}{\partial t^2} G_Y(z,t) \Biggr|_{t=1}
    +\frac{\partial}{\partial t} G_Y(z,t) \Biggr|_{t=1}\\
    &=\frac{z(2(1+z)-(1+z)^2e^{-z}+e^{-2z})}{2(1-z)^3}
    + \frac{\sqrt{2\pi z}}{4(1-z)^2}\,e^{-z/2}
    (1+z^2-e^{-z}).
\end{align*}
It follows that
\begin{align*}
    S_X(z) := \sum_{n\ge0}\mathbb{E}(X_n^2)z^n
    = \int_0^z\frac{f_2(u)}{(1-u)^4}\, \dd u,
\end{align*}
where
\begin{align*}
    f_2(z) &= \tfrac12f_1(z)^2-e^{-z}f_1(z)
    +(z^2-z+2)f_1(z) \\ &\qquad +(1-z)(1+2z)e^{-z}
    -\tfrac12(1-z)^2(3+2z).
\end{align*}
Consider now $[z^n]f_2(z)$. By \eqref{f1z-asymp}, we see that
the first two terms on the right-hand side dominate and
contribute an order bounded above by
\[
    [z^n]\left(\tfrac12f_1(z)^2-e^{-z}f_1(z)\right)
    \le 4\sum_{2\le k\le n-2}\frac1{k!(n-k)!}
    = O\left(\frac{2^n}{n!}\right),
\]
the remaining terms being of order $O(1/n!)$. Thus
\[
    [z^n]f_2(z) = O\left(\frac{2^n}{n!}\right).
\]
By another application of \eqref{fQz2}, we derive an asymptotic
approximation to the second moment with an error term of the form
$O(2^n/(n+4)!)$, which, together with \eqref{EXn2}, proves
\eqref{VXn}.
\end{proof}

\begin{figure}[!h]
\begin{center}
\includegraphics[width=5.5cm]{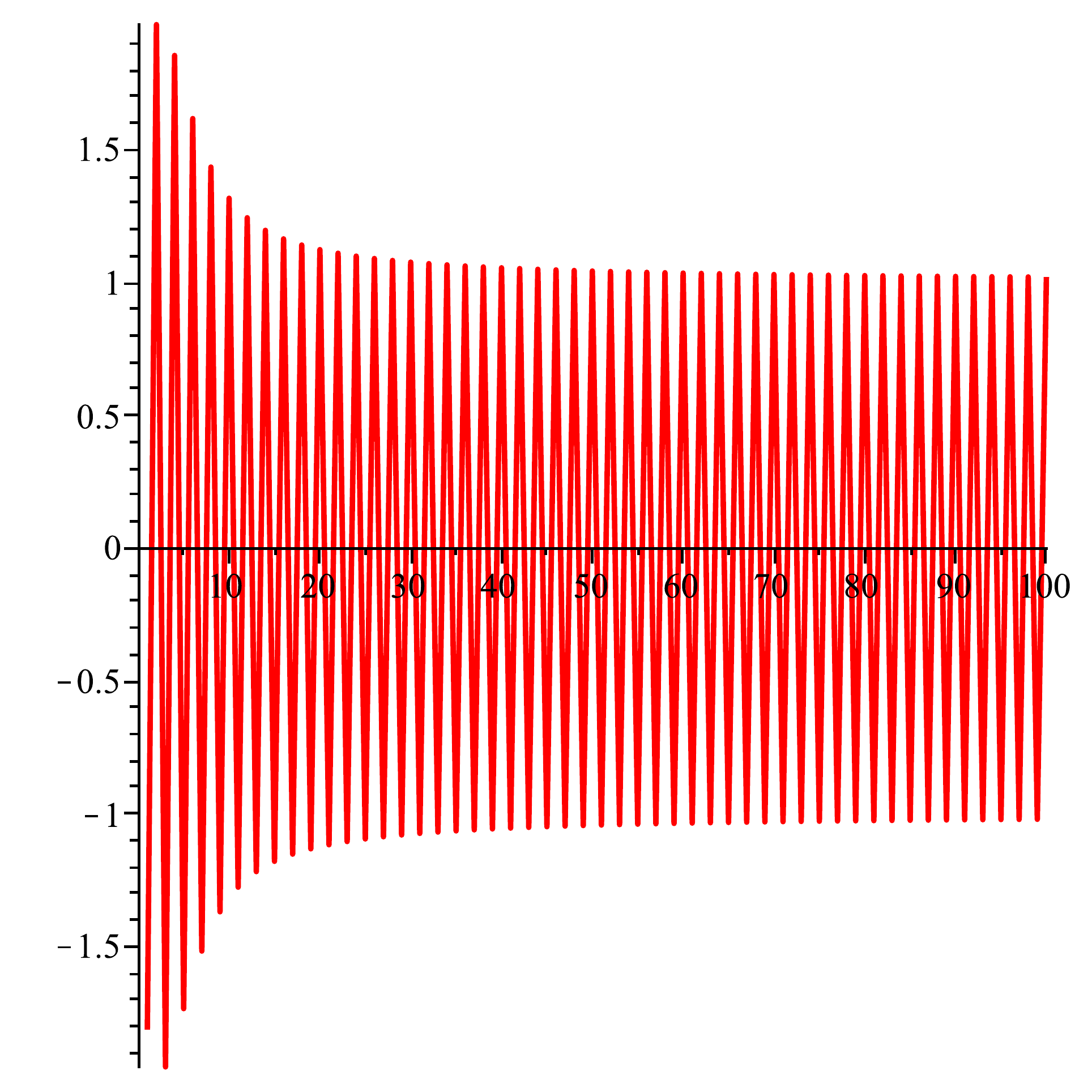}\;\;
\includegraphics[width=5.5cm]{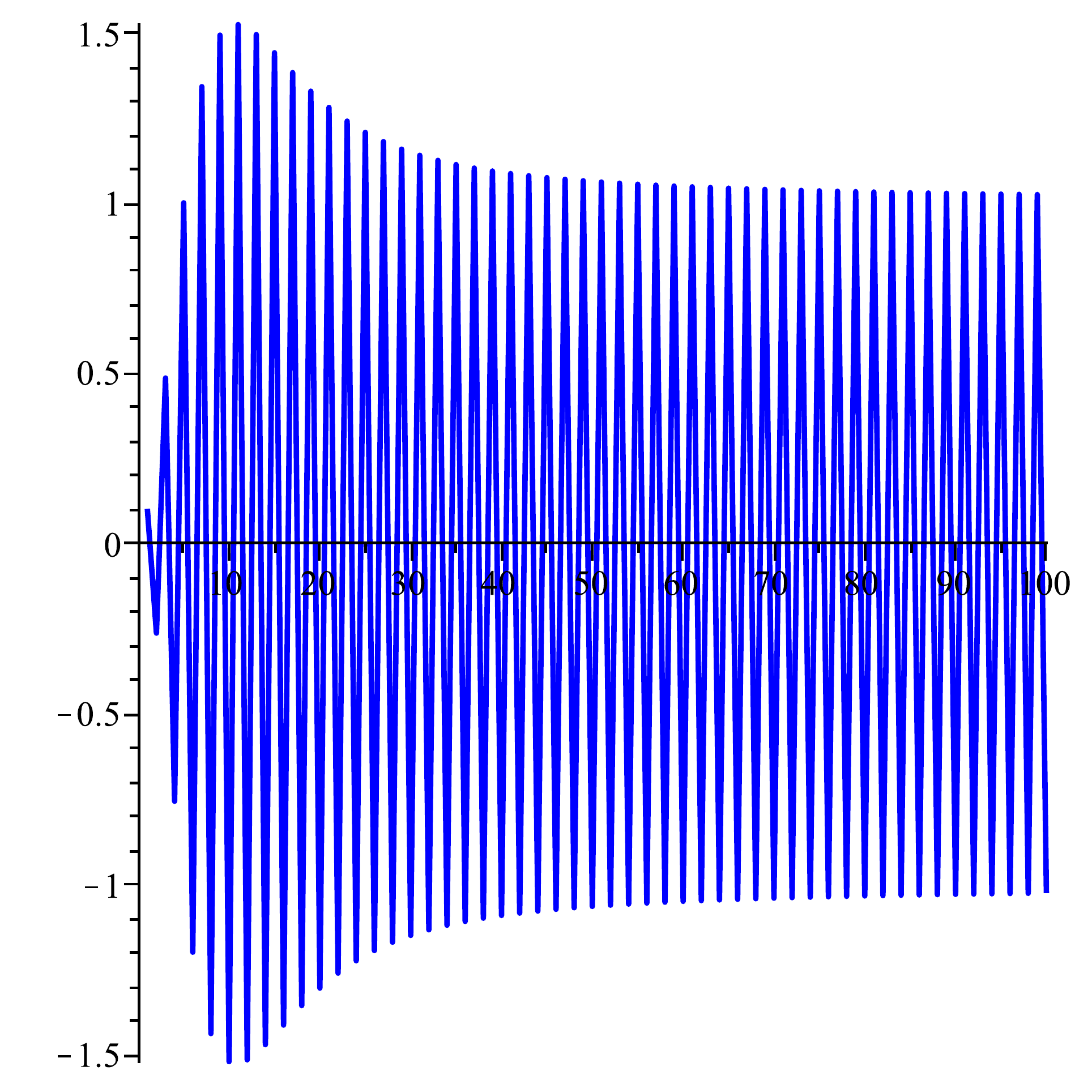}
\end{center}
\vspace*{-.5cm}
\caption{\emph{The factorial errors of \eqref{EXn2} and
\eqref{VXn}: $(\mathbb{E}(X_n)-\mu n - c_1)(n+3)!/2$ (left)
and $(\mathbb{V}(X_n)-\sigma^2 n - c_2)(n+4)!/2^{n+5}$ (right)
for $n=1,\dots,100$.}}
\label{fg-EV2}
\end{figure}

\section{An identity for $X_n(t)$}

Since solutions to Riccati equations have only simple poles, we
expect, from the closed-form expression \eqref{FY-PQ}, that
\begin{align}\label{GYz-pf}
    G_Y(z,t) = \frac{Q(z,t)}{P(z,t)}
    \approx \sum_{k\in\mathbb{Z}} \frac{R_k(t)}{\rho_k(t)-z},
\end{align}
where $\rho_k(t)$ ranges over all zeros of $P(z,t)$ (as a function
of $z$) and
\[
    R_k(t) := -\frac{Q(\rho_k(t),t)}{P'(\rho_k(t),t)},
\]
\emph{Here and throughout this section} $P'(z_0,t) =
(\partial/\partial z)P(z,t)|_{z=z_0}$. The expansion \eqref{GYz-pf}
is roughly true up to correction terms in the series to guarantee
convergence; see \eqref{FYzt}. From this series, we in turn expect
that
\[
    Y_n(t) = \mathbb{E}\left(t^{Y_n}\right)
    \stackrel{?}{=} \sum_{k\in\mathbb{Z}} R_k(t) \rho_k(t)^{-n-1},
\]
which is indeed true for $n\ge1$; see \eqref{Ynt}. What is less
expected is that their convolution \eqref{Xnt-rr}, which yields
$X_n(t)$, also admits a closed-form expression.

Before stating the identity for $X_n(t)$, we start with a brief
discussion for the zeros of $P(z,t)$, namely,
\[
    (1-tz)e^{tz} = \frac{1-t}{1+t},
\]
which are easily seen to be expressible in terms of Lambert's
W-functions \cite{Corless96}. They are the solutions to the equation
\begin{align}\label{wew}
    W(z) e^{W(z)} = z.
\end{align}
This equation has an infinity number of solutions $W_k(z)$,
$k\in\mathbb{Z}$, and among them only one, denoted by $W(z)=W_0(z)$,
is analytic at the origin. This function has the Taylor series
expansion
\begin{align}\label{Wz}
    W(z) = -\sum_{k\ge1}\frac{k^{k-1}}{k!}\,(-z)^k,
\end{align}
and has the branch cut $(-\infty,-e^{-1})$. All other solutions have
the branch cut $(-\infty,0]$.

With these solutions, we have $P(\rho_k(t),t) = 0$ when
\[
    \rho_k(t) = \frac1t\left(1+W_k\left(-e^{-1}
    \frac{1-t}{1+t}\right)\right),
\]
where $\rho_0^{}(t)$ has the branch cut $[-1,0]$ and the other
branches the cut $[-1,1]$. As $t\to1$, all solutions blow up to
infinity except for $\rho_0^{}$ which equals $1$ at $t=1$.

A useful expansion that will be needed is the following convergent
series (see \cite{Corless96})
\[
    W_k(z) = \log z + 2k\pi\bfi -\log(\log z + 2k\pi\bfi)
    +\sum_{j\ge0} \frac{\Pi_j(\log(\log z + 2k\pi\bfi))}
    {(\log z + 2k\pi\bfi)^j},
\]
valid for all $z$, where $\Pi_j(x)$ is a polynomial in $x$ of degree
$j$. In particular, this gives for finite $z$ and $n\ne0$
\begin{align}\label{Wkz}
    |W_k(z)| = O(k+|\log z|).
\end{align}

\begin{thm} For $n\ge1$, we have the identity \label{thm:Xnt}
\begin{align}\label{Xnt}
    X_n(t) = t\sum_{k\in\mathbb{Z}} R_k(t)^2
    \rho_k(t)^{-n-1},
\end{align}
for $n\ge1$ and $t\in\mathbb{C}\setminus \{-1\}$, where
\[
    R_k(t) := \frac1t\left(1-\frac{1-t}{2(1+t)}
    \int_0^1 v^{-1/2} e^{-t\rho_k(t)(1+v)/2} \dd v \right).
\]
When $t=-1$, we have the identity
\begin{align}\label{Xnm1}
    X_n(-1) = -\frac{(-2)^{n-1}}{n} \sum_{0\le k<n}
    \frac{k!(n-1-k)!}{(2k)!(2n-2-2k)!},
\end{align}
and the asymptotic approximation
\begin{align}\label{Xnm1-asymp}
    X_n(-1)= 2\frac{n!(-4)^n}{(2n)!\sqrt{\pi n}}
    \left(1+\frac{9}{8n}+O\left(n^{-2}\right)\right).
\end{align}
\end{thm}
The expression \eqref{Xnt} is not only an identity but also an
asymptotic expansion for large $n$ (finite $t$). The left-hand side
is by definition a polynomial of degree $n$, while the right-hand
side is an infinite series of exponentially decreasing terms. It
implies particularly that $X_n(t)$ is roughly of the exponential
order $|\rho_0^{}(t)^{-n}|$ except when $t=-1$ at which $X_n(t)$ is
factorially small. Although $R_k$ can be further expressed in terms
of known functions, the expression we give here is more transparent
and valid for all $t\in\mathbb{C}\setminus\{-1\}$.

\begin{proof}
We start from the local expansion
\[
    G_Y(z,t) \sim \frac{R(\rho(t),t)}{\rho(t)-z},
\]
as $z\sim \rho(t)$, where $P(\rho(t),t)=0$ and
\[
    R(z,t) := -\frac{Q(z,t)}{P'(z,t)}
    = \frac1t\left(1-\frac{1-t}{2(1+t)}
    \int_0^1 v^{-1/2} e^{-tz(1+v)/2} \dd v \right).
\]
A more precise expansion is given as follows
\begin{align*}
    G_Y(z,t) &= \frac{R(\rho(t),t)}{\rho(t)-z}
    + \frac{Q'(\rho(t),t)}{P'(\rho(t),t)}
    - \frac{Q(\rho(t),t)P''(\rho(t),t)}
    {2P'(\rho(t),t)^2} + O(|z-\rho(t)|)\\
    &= \frac{R(\rho(t),t)}{\rho(t)-z}+ O(|z-\rho(t)|),
\end{align*}
where the constant term turns out to be identically zero because
\begin{align}\label{PQ-cst}
\begin{split}
    &2P'(z,t)Q'(z,t)-Q(z,t)P''(z,t)
    = \frac{t^2(1-t)}2\,P(z,t)
    \int_0^1 v^{-1/2} e^{-tz(1+v)/2} \dd v.
\end{split}
\end{align}
This is crucial in proving \eqref{Xnt}.

Since all zeros of the $P(z,t)$ are simple, we have the partial
fraction expansion
\begin{align} \label{FYzt}
    G_Y(z,t) = 1+ \sum_{j\in\mathbb{Z}}R_j(t)
    \left(\frac1{\rho_j(t)-z}-\frac1{\rho_j(t)}\right),
\end{align}
by the classical procedure for meromorphic functions (see \cite[\S
3.2]{Titchmarsh39}), where we used the estimate \eqref{Wkz} for
$W_k$ (see \cite{Corless96}) and the asymptotic approximation
\[
    2\Phi(\sqrt{x})-1 \sim 1 - \sqrt{\frac 2\pi}\,x^{-1/2}
    e^{-x/2}\qquad(x\to\infty).
\]

This implies the identity
\begin{align}\label{Ynt}
    Y_n(t) = \sum_{j\in\mathbb{Z}}R_j(t)\rho_j(t)^{-n-1}
    \qquad(n\ge1).
\end{align}

To prove \eqref{Xnt}, we start with the convolution \eqref{Xnt-rr}
\begin{align*}
    X_n(t) &= \frac{2t}{n}\,Y_{n-1}(t) + \frac{t}{n}
    \sum_{1\le k\le n-2}Y_k(t) Y_{n-1-k}(t) \\
    &= \frac{2t}{n}\sum_{j\in\mathbb{Z}}
    R_j\rho_j^{-n} + \frac{t}{n}
    \sum_{j,\ell\in\mathbb{Z}} R_jR_\ell
    \sum_{1\le k\le n-2} \rho_j^{-k-1}\rho_{\ell}^{-n+k},
\end{align*}
where for convention we drop the dependence on $t$. By the relation
\[
    \sum_{1\le k\le n-2} x^{-k-1}y^{-n+k}
    = \begin{cases}
        (n-2) x^{-n-1},& \text{if }x=y\\
        \displaystyle \frac{x^{-1}y^{-n+1}-y^{-1}x^{-n+1}}{x-y},
        &\text{if } x\not=y
    \end{cases},
\]
we then have
\begin{align*}
    &\frac{t}{n}\sum_{j,\ell\in\mathbb{Z}} R_jR_\ell
    \sum_{1\le k\le n-2} \rho_j^{-k-1}\rho_{\ell}^{-n+k}\\
    &\quad= \frac{t}{n}(n-2)\sum_{j\in\mathbb{Z}}
    R_j^2\rho_j^{-n-1} + \frac{t}{n}\sum_{j\in\mathbb{Z}}
    \sum_{\ell\not= j} R_jR_\ell
    \frac{\rho_j^{-1}\rho_\ell^{-n+1}-
    \rho_\ell^{-1}\rho_j^{-n+1}}
    {\rho_j-\rho_\ell}\\
    &\quad= \frac{t}{n}(n-2)\sum_{j\in\mathbb{Z}}
    R_j^2\rho_j^{-n-1} + \frac{2t}{n}\sum_{j\in\mathbb{Z}}
    R_j \rho_j^{-n+1}\sum_{\ell\not= j} \frac{R_\ell}
    {\rho_\ell(\rho_\ell-\rho_j)}.
\end{align*}
Then we have
\begin{align*}
    X_n(t) &= t\sum_{j\in\mathbb{Z}} R_j^2
    \rho_j^{-n-1}  \\
    &\quad+ \frac{2t}{n}
    \Biggl\{ \sum_{j\in\mathbb{Z}} \left(\rho_j
    R_j-R_j^2\right) \rho_j^{-n-1}
    + \sum_{j\in\mathbb{Z}}
    R_j \rho_j^{-n+1}\sum_{\ell\not= j} \frac{R_\ell}
    {\rho_\ell(\rho_\ell-\rho_j)}\Biggr\}.
\end{align*}
The last double-sum can be further simplified. For, by \eqref{FYzt},
\[
    \lim_{z\to\rho_j} \left(G_Y(z,t) - \frac{R_j}{\rho_j-z}
    \right) = 1-\frac{R_j}{\rho_j} + \sum_{\ell\ne j}
    \frac{\rho_jR_\ell}{\rho_\ell(\rho_\ell-\rho_j)},
\]
on the one hand, and, by \eqref{PQ-cst},
\[
    \lim_{z\to\rho_j} \left(G_Y(z,t) - \frac{R_j}{\rho_j-z}
    \right) = 0,
\]
on the other hand. It follows that
\begin{align*}
    \sum_{\ell\ne j}
    \frac{R_\ell}{\rho_\ell(\rho_\ell-\rho_j)}
    &= -\frac{1}{\rho_j}+\frac{R_j}{\rho_j^2}.
\end{align*}
Thus
\[
    \sum_{j\in\mathbb{Z}} \left(\rho_j
    R_j-R_j^2\right) \rho_j^{-n-1}
    + \sum_{j\in\mathbb{Z}}
    R_j \rho_j^{-n+1}\sum_{\ell\not= j} \frac{R_\ell}
    {\rho_\ell(\rho_\ell-\rho_j)} = 0,
\]
and we conclude the identity \eqref{Xnt}.

Consider now $t=-1$ at which $X_n(t)$ satisfies
\[
    X_n(-1) = \sum_{n/4\le k\le n/2} \left(\mathbb{P}(X_n=2k)
    - \mathbb{P}(X_n=2k-1)\right).
\]
By \eqref{FY-PQ}, we have
\begin{align*}
    G_Y(z,-1)
    = 1-\frac z2\int_0^1 v^{-1/2} e^{-(1-v)z/2}\dd v,
\end{align*}
which, by a direct expansion of the exponential factor and
term-by-term integration, implies that
\[
    Y_n(-1) 
    = \frac{n!(-2)^n}{(2n)!}\qquad(n\ge0).
\]
From this we derive \eqref{Xnm1}. Express now the convolution sum
\eqref{Xnm1} as an integral as follows
\[
    X_n(-1) = \frac{(n-1)!(-2)^{n-2}}{(2n-2)!}
    \int_0^1 \left((1+2\sqrt{v(1-v)})^{n-1}
    +(1-2\sqrt{v(1-v)})^{n-1}\right) \dd v ,
\]
where we used the relation
\[
    \sum_{0\le k\le n} \binom{2n}{2k}
    z^k = \frac12\left((1+z+2\sqrt{z})^n
    +(1+z-2\sqrt{z})^n\right).
\]
Then the asymptotic expression \eqref{Xnm1-asymp} follows from a
simple application of the saddle-point method. This completes the
proof of the theorem.
\end{proof}

\section{Approximation theorems}

The identity \eqref{Xnt}, when viewing as an asymptotic expansion,
is very useful in deriving limit and approximation theorems with
optimal convergence rate, following the Quasi-Power Framework; see
\cite[\S IX. 5]{FS09},\cite{Hwang98}. Other properties such as
moderate and large deviations can also be derived by standard
arguments.

We start from the ``Quasi-Power approximation"
(see \emph{loc.\ cit.})
\[
    \mathbb{E}\left(e^{X_n s}\right)
    = e^sR_0^2(e^s)\rho(e^s)^{-n-1}
    \left(1+O\left(\ve^n\right)\right),
\]
for some $\ve>0$, uniformly for $|s|\le \delta$ in a small
neighborhood of origin. The exact values of $\ve$ and $\delta$ can
be made explicit by numerical calculations and standard Rouch\'e's
theorem. For example, if we take $\delta=0.2$, then $\ve=1/2$
suffices; see \cite{FGM97} for a similar context.
\begin{figure}
\begin{center}
\includegraphics[width=5cm]{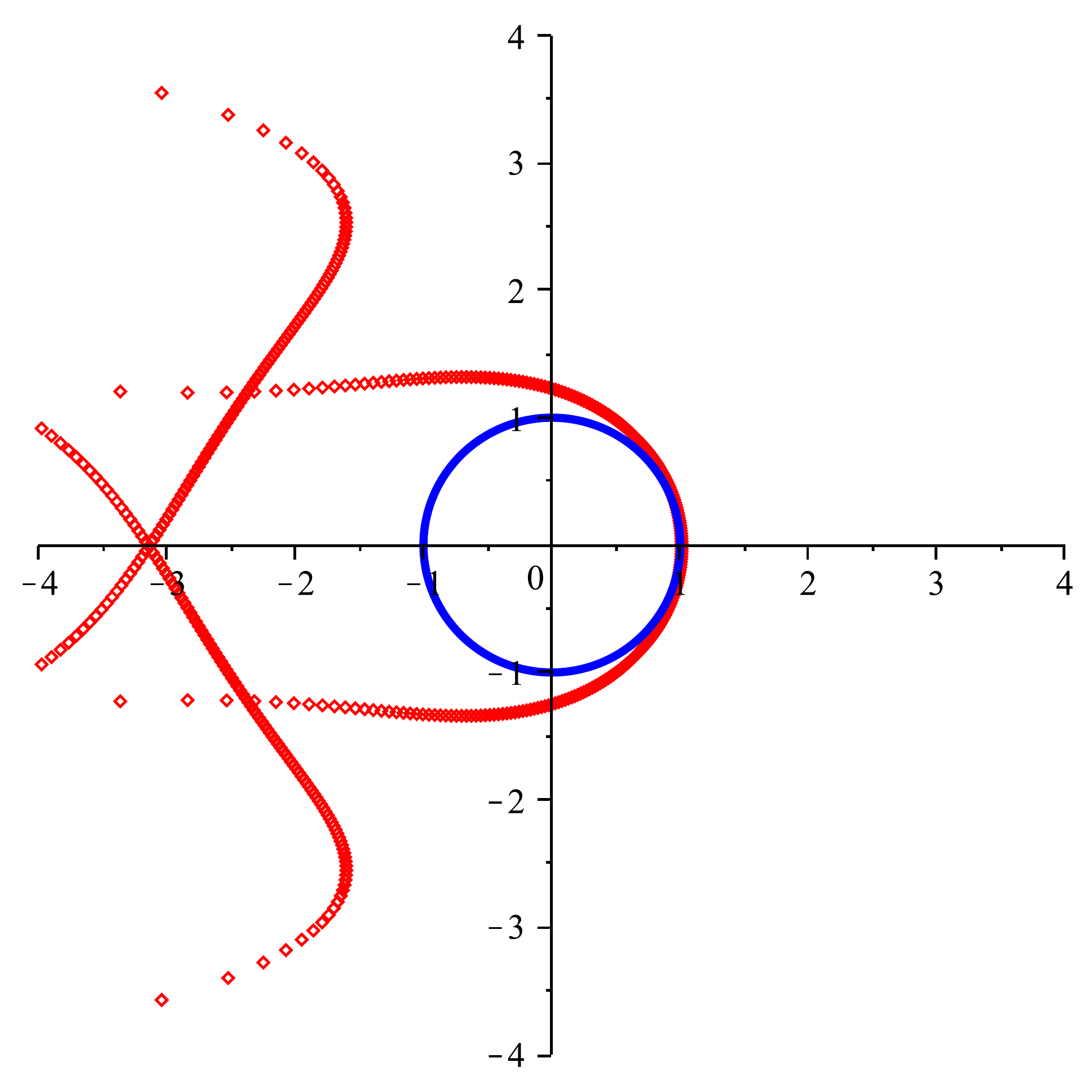}\;\;
\includegraphics[width=5cm]{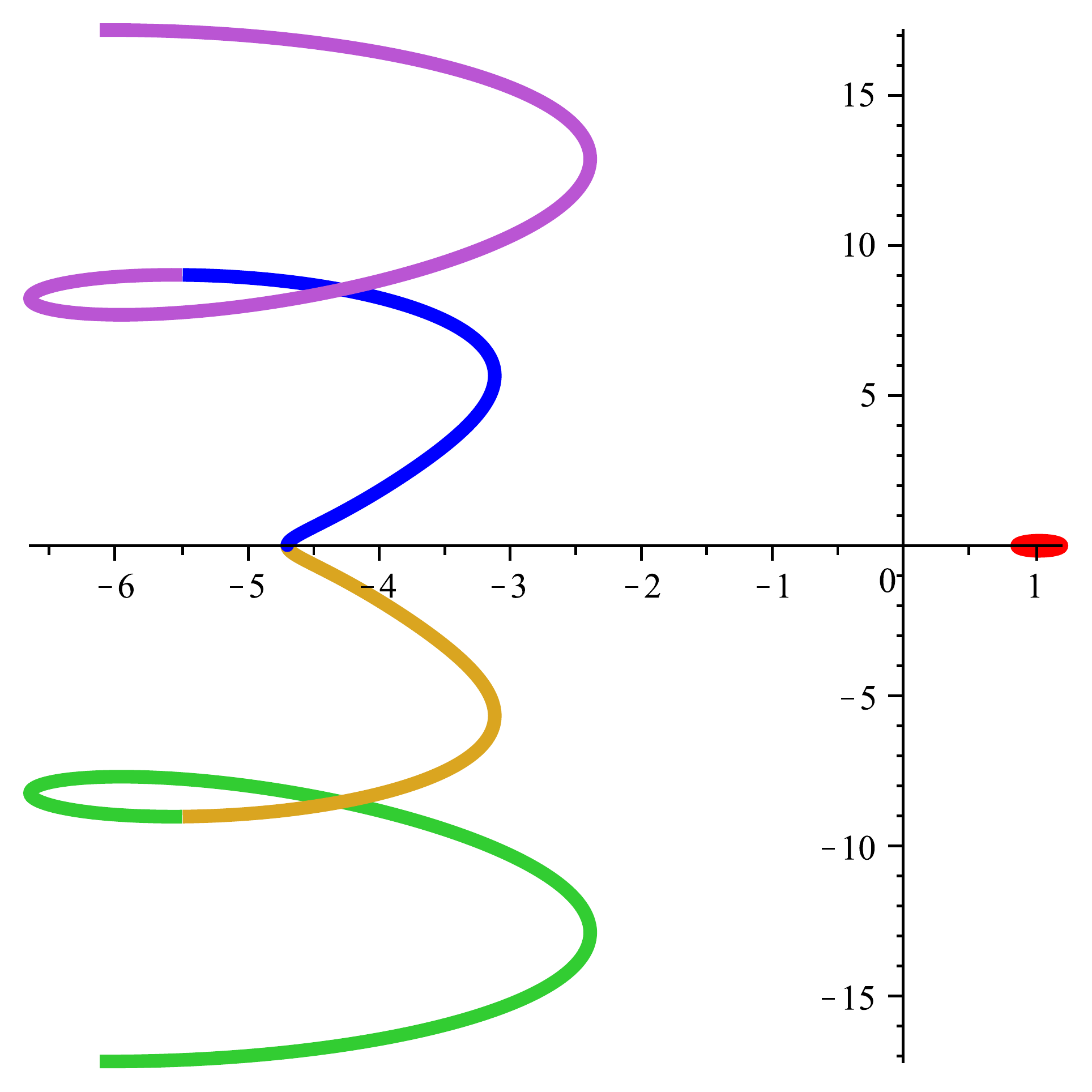}\;\;
\caption{\emph{Approximate zeros of the denominator
$P(z,e^{\bfi\theta})$ of $G_Y$ inside the rectangular region
$[-4-4i, 4+4i]$ (left), and the fives curves
$\{\rho_j(1+0.2e^{\bfi\theta})\}_{j=-2}^2$
for $-\pi\le\theta\le\pi$ (left), where $\rho_0^{}$
is the small red circle near unity.}}
\end{center}
\end{figure}
From this approximation and by a direct Taylor expansion of
$-(n+1)\log\rho(e^s) +s+2\log(R_0(e^s))$ (and justified by the
Quasi-Power Framework; see \emph{loc.\ cit.}), we obtain the two
dominant terms in \eqref{EXn2} and \eqref{VXn} (with weaker error
terms). Moreover, the same argument applies for higher central
cumulants (or moments). In particular, the third and fourth
cumulants are asymptotic to
\[
    \frac{e^{-3}}8\left(2e^2-17\right)n + c_3,
    \quad \text{and} \quad
    \frac{e^{-4}}8\left(-12e^3+71\right) n + c_4,
\]
respectively, where ($\phi := 2\Phi(1)-1$){\small
\begin{align*}
    c_3 &= \frac{e^{-3}}{16}\left(
    (2\pi e)^{3/2}\phi^3 + 12 \pi e \phi^2
    -\sqrt{2\pi e}\left(4e^2-15\right)\phi
    -64+4e^2\right)\\
    &\approx -0.01646\,99733\,69929\dots\\
    c_4 &= \frac{e^{-4}}{16}\left(-3(e\pi)^2\phi^4
    -6(2\pi e)^{3/2}\phi^3 + 2 \pi e(4e^2-21) \phi^2
    +4\sqrt{2\pi e}\left(4e^2-11\right)\phi
    +280-40e^2\right)\\
    & \approx 0.09122\,16766\,24710\dots
\end{align*}}
These expressions show the strength of the Quasi-Power approach.
Although the direct approach used in Section~\ref{sec:mv} to compute
the first two moments provides more precise error terms (factorial
instead of exponential), the approach used here is computationally
simpler, notably for the expressions of the constant terms of
high-order central cumulants.

\begin{figure}
\begin{center}
\includegraphics[width=5.5cm]{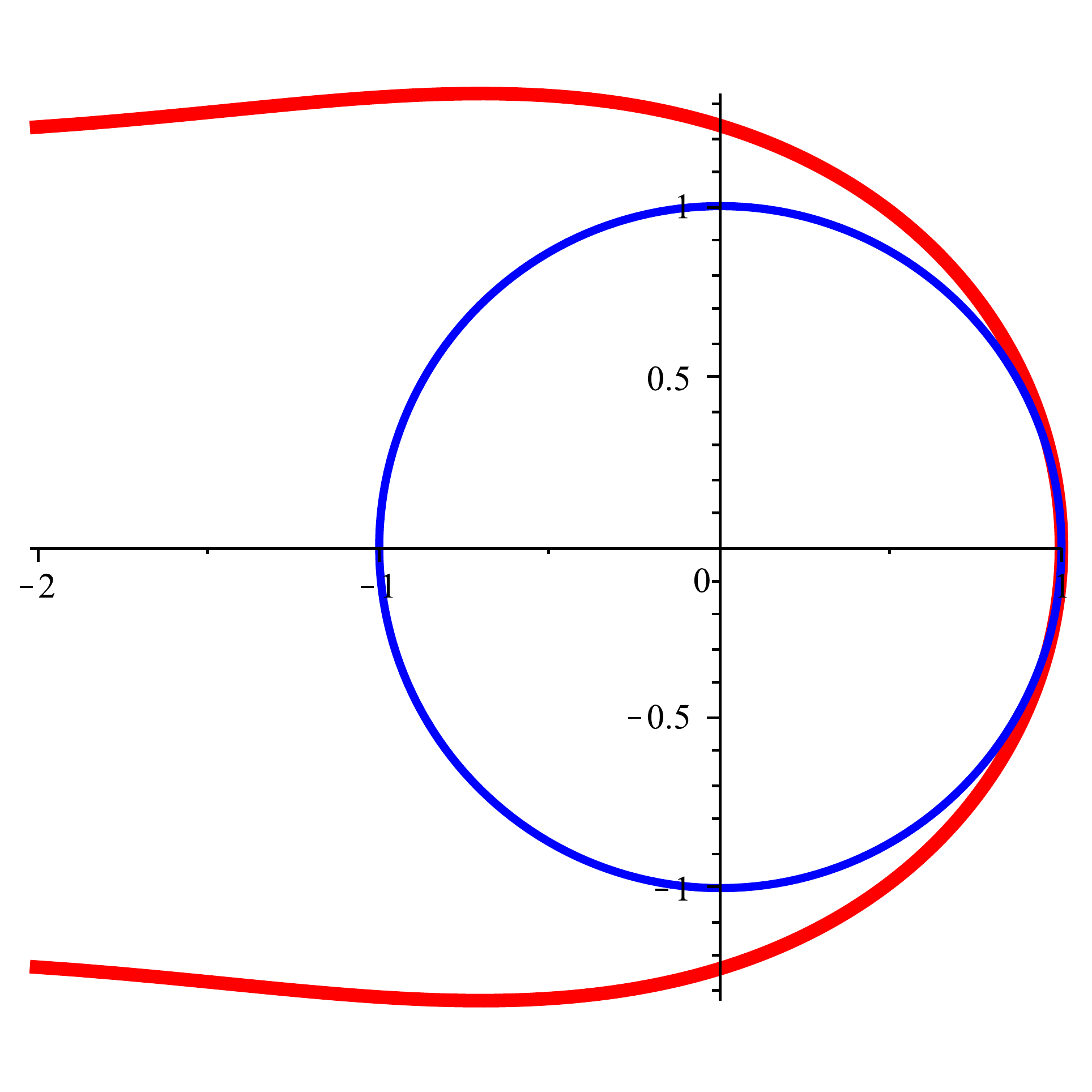}\;\;
\includegraphics[width=5.5cm]{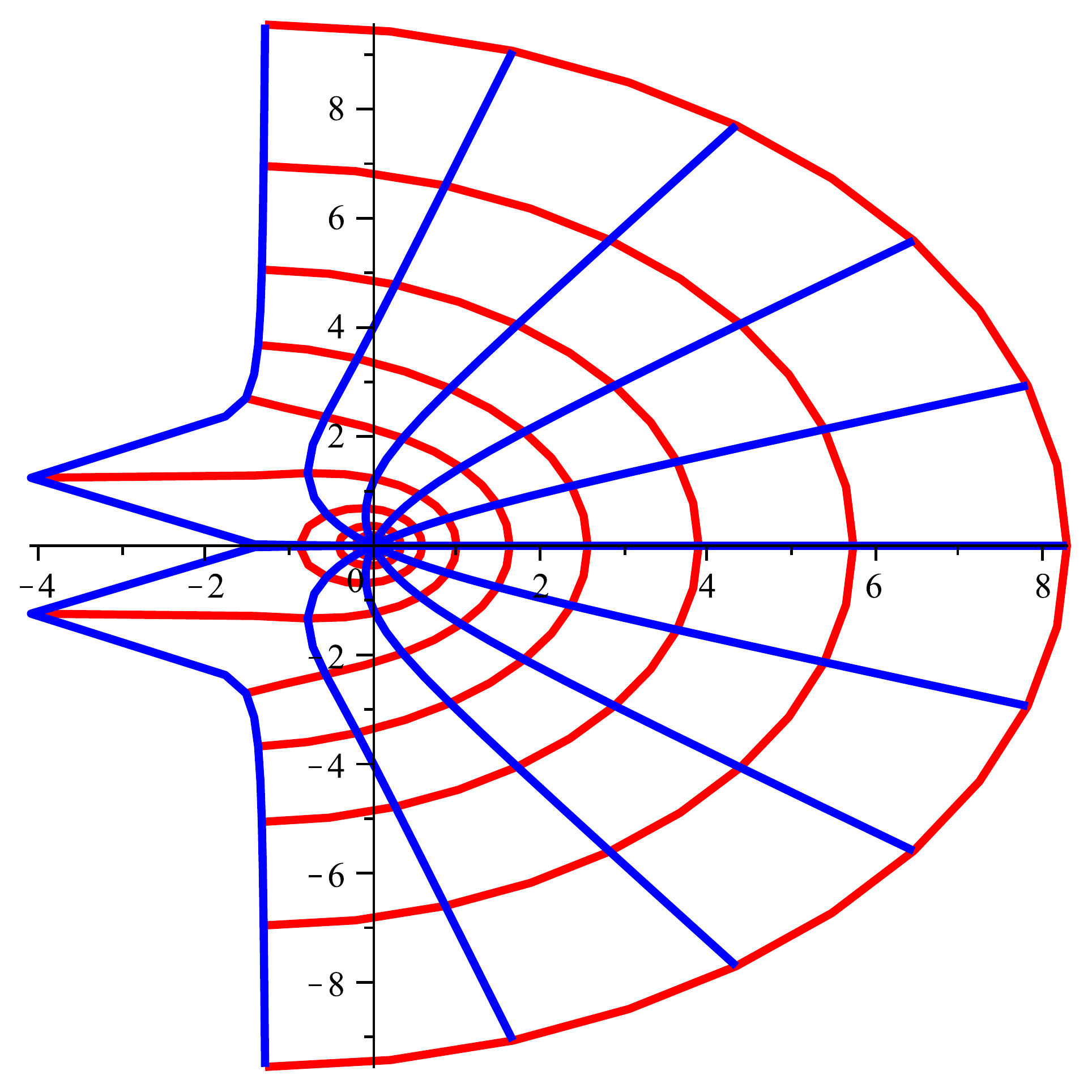}
\caption{\emph{The curve $\rho(t)$ when $|t|=1$ (left),
where the unit circle is also shown, and a conformal plot of
$\rho(e^w)$ (right).}}
\end{center}
\end{figure}

For limit and approximation theorems, we are particularly interested
in the behavior of the dominant term $\rho(t) := \rho_0^{}(t)$ in
the asymptotic expansion \eqref{Xnt} when $|t|=1$. Note that
$\rho(1)=1$ and all other $\rho_k(-e^{-1}(1-t)/(1+t))$'s tend to
infinity when $t\to1$. Also
\[
    \rho_k(e^{\bfi\theta}) = e^{-\bfi\theta}
    \left(1+W_k\left(e^{-1}
    \frac{\sin\theta}{1+\cos\theta}\,\bfi\right)\right).
\]

From \eqref{Xnt}, we have, when $|t|=1$
\begin{align} \label{Xnt-ub}
    X_n(t) = t R(t)^2 \rho(t)^{-n-1}
    + O\left(4^{-n}\right).
\end{align}

\begin{thm}[Central and local limit theorems] Let $\mu := 1-e^{-2}/2$
and $\sigma=\sqrt{\frac34}\,e^{-1}$. We have
\begin{align}\label{BE-bd}
    \sup_{x\in\mathbb{R}}
    \left|\mathbb{P}\left(\frac{X_n-\mu n}{\sigma \sqrt{n}} \le x
    \right) - \Phi(x)\right| = O\left(n^{-1/2}\right),
\end{align}
and, uniformly for $x=o(n^{1/6})$,
\begin{align}\label{LLT}
    \mathbb{P}\left(X_n = \tr{\mu n + x\sigma \sqrt{n}}\right)
    = \frac{e^{-x^2/2}}{\sqrt{2\pi n}\, \sigma} \left(1+
    O\left((1+|x|^3)n^{-1/2}\right)\right).
\end{align}
\end{thm}
\begin{proof} (Sketch)
The convergence rate \eqref{BE-bd} follows from \eqref{Xnt-ub} and
the classical Berry-Esseen inequality, and is part of the
Quasi-Power Theorem (see \emph{loc.\ cit.}). The local limit theorem
is also straightforward by the corresponding Fourier integral
representation once we have the uniform bound \eqref{Xnt-ub}.
Details are omitted.
\end{proof}
Note that the Berry-Esseen bound \eqref{BE-bd} with a rate of the
form $n^{-1/2+\ve}$ was established in \cite{PS05}; their
formulation is more general but with slightly less precise
approximations.

\begin{figure}
\begin{center}
\includegraphics[width=5cm]{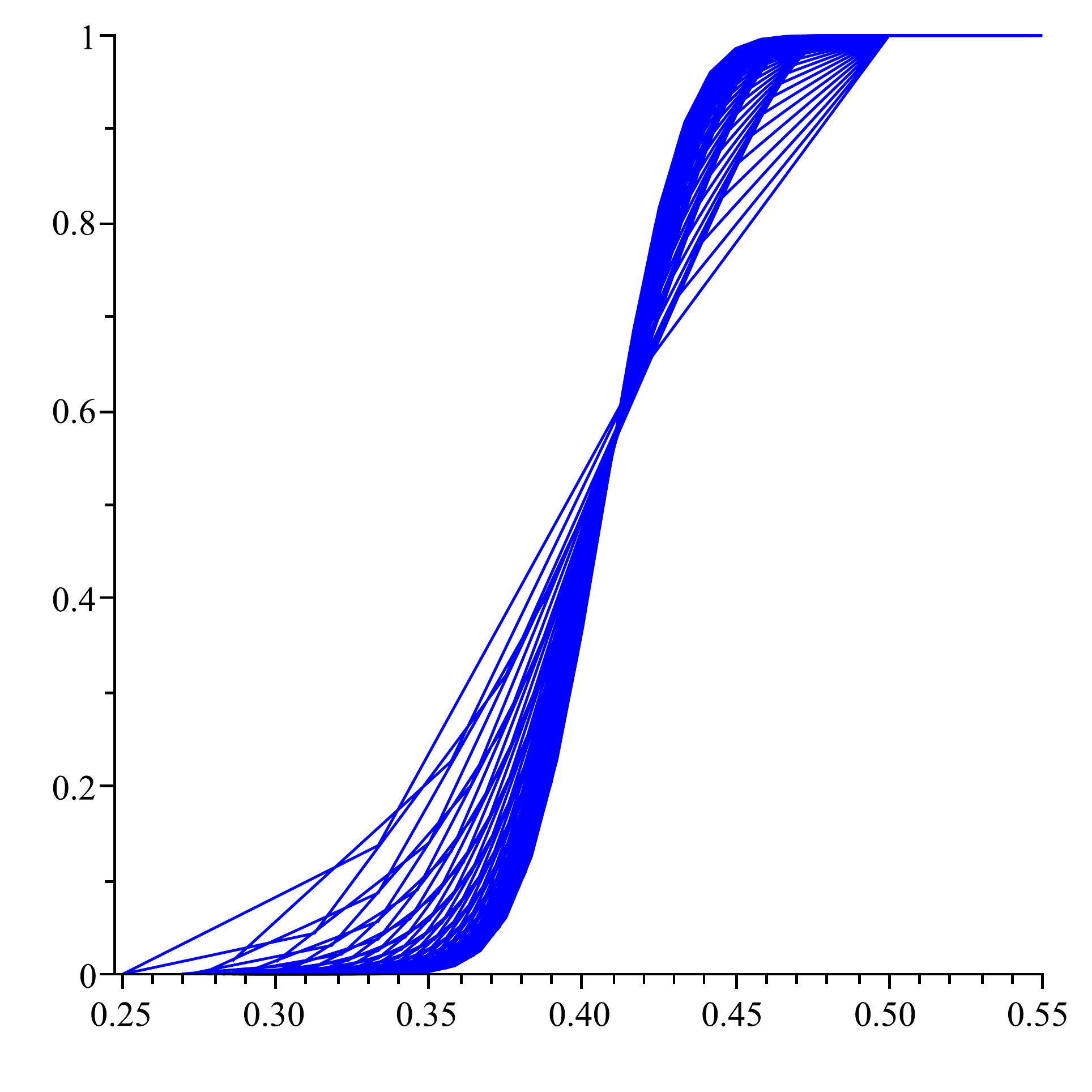}\;\;
\includegraphics[width=5cm]{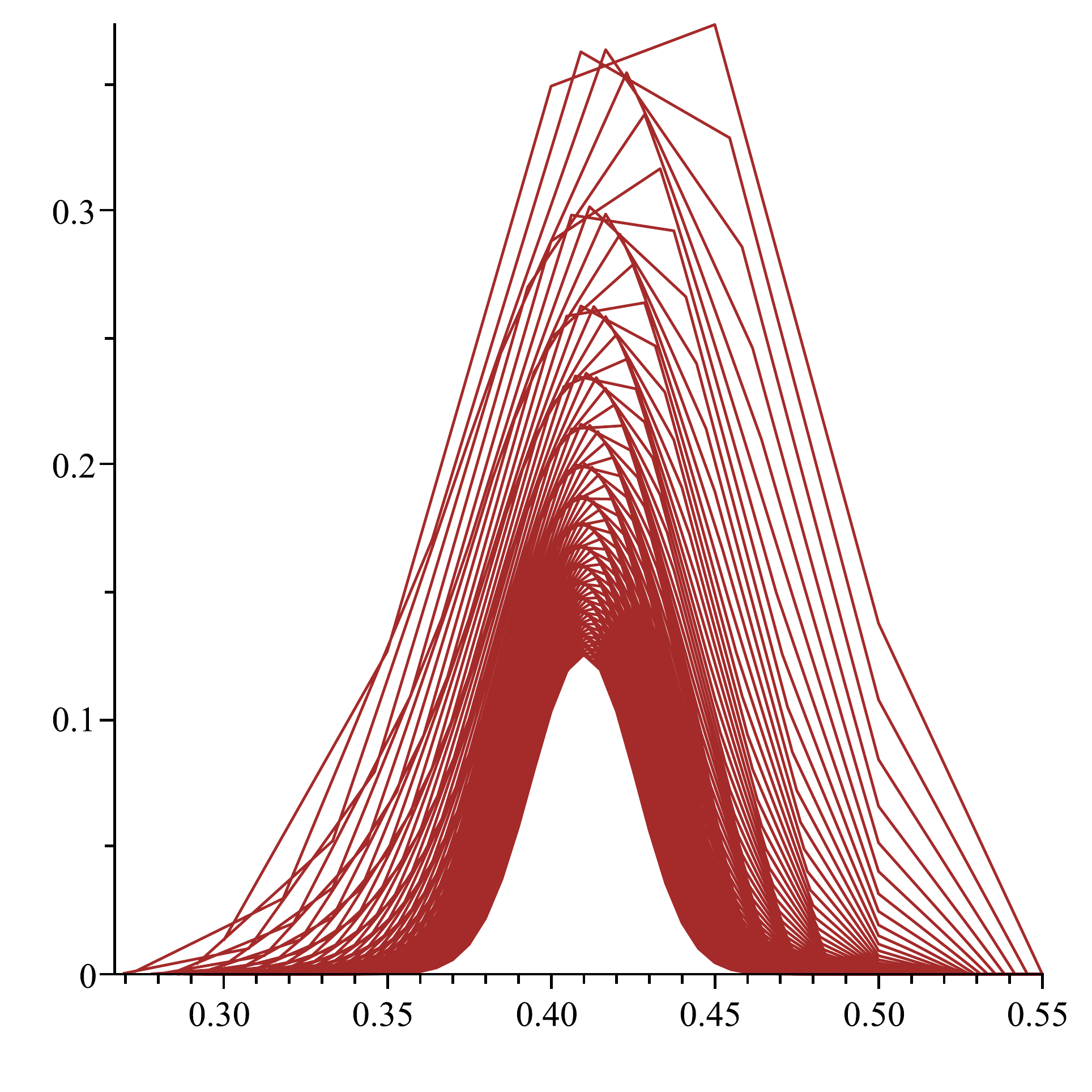}
\caption{\emph{Exact distributions of $X_n$ for $n=6,\dots,60$:
the distributions are plotted against $1/2n$.}}
\end{center}
\end{figure}

\section{Stochastic dominance}
\label{sec:sd}

We clarify the following stochastic dominance relations in this
section.
\begin{thm} For $n\ge1$ \label{thm-ABX}
\begin{equation}
    A_{n+1},B_{n+1}\ge X_{n}\ge A_{n-1}-2,B_{n-1}-2,
    \label{w1}
\end{equation}
where we write $X\ge Y$ (in distribution) if for all $x$
\[
    \mathbb{P}(X\le x) \le \mathbb{P}(Y\le x).
\]
\end{thm}
So the asymptotic normality of $X_n$ can be reduced to that of $A_n$
and $B_n$, which is easier because of the simpler recurrences or the
closed-form expressions \eqref{FAFB}.

The sandwich approximation (\ref{w1}) seems intuitively clear but a
rigorous proof is far from being obvious. Our proof given below is
simple but messy. On the other hand, the ``$-2$'' factors in
\eqref{w1} are not optimal and might be replaced by ``$-1$''; but
our proof is somewhat too weak to justify this.

To prove (\ref{w1}), we establish the following dependence graph of
stochastic dominance relations.
\begin{description}
\item[(i)]
\[
    \text{(0a) }1+B_{n}\ge A_{n},\quad
    \text{(0b) }1+A_{n}\ge B_{n},
\]
\item[(ii)]
\[
    \text{(1a) }A_{n}\ge A_{n-1},\quad
    \text{(1b) }B_{n}\ge B_{n-1},\quad
    \text{(1c) }Y_{n}\ge Y_{n-1},
\]
\[
    \ \ \text{(2a) }A_{n}\ge Y_{n},\ \ \ \quad
    \text{(2b) }B_{n}\ge Y_{n},\quad \ \ \
    \text{(2c) }Y_{n}\ge X_{n-1},
\]
\item[(iii)]
\[
    \text{(3a) }1+B_{n-1}\ge A_{n},\quad \ \
    \text{(3b) }1+A_{n-1}\ge B_{n},\quad
    \text{(3c) }1+Y_{n-1}\ge Y_{n},
\]
\[
    \text{(4a) }1+A_{n-1}\ge Y_{n},\quad \ \ \
    \text{(4b) }1+B_{n-1}\ge Y_{n},\quad \ \
    \text{(4c) }1+Y_{n}\ge X_{n},
\]
\item[(iv)]
\[
    \text{(5a) }1+Y_{n}\ge A_{n-1},\quad \ \
    \text{(5b) }1+Y_{n}\ge B_{n-1},\quad \ \
    \text{(5c) }1+X_{n}\ge Y_{n}.
\]
\end{description}

Combining (2a), (2b) and (2c), we obtain the left-hand side
of \eqref{w1}
\[
    A_{n},B_{n}\ge X_{n-1};
\]
on the other hand, combining (5a), (5b) and (5c) leads to
\[
    2+X_{n}\ge A_{n-1},B_{n-1}.
\]
which is the right-hand side of \eqref{w1}.

The following directed graph indicates the implications of the
diverse stochastic dominance relations. The symbol ``A $\rightarrow
$ B" means that the proof of B uses the induction hypothesis of A.

\quad

\begin{center}
\begin{tikzpicture}[scale=.7, transform shape]
\tikzstyle{every node} = [circle, draw=black,thick]

\node (0a) at (3, 6) {0a,0b};
\node (2a) at (-6, 3) {2a,2b};
\node (2c) at (-6, 0) {\ \ 2c \ \ };
\node (4a) at (-3, 3) {4a,4b};
\node (4c) at (-3, 0) {\ \ 4c \ \ };
\node (1a) at (0, 3) {1a,1b};
\node (1c) at (0, 0) {\ \ 1c \ \ };
\node (3a) at (3, 3) {3a,3b};
\node (3c) at (3, 0) {\ \ 3c \ \ };
\node (5a) at (6, 3) {5a,5b};
\node (5c) at (6, 0) {\ \ 5c \ \ };
\foreach \from/\to in {1a/1c, 2a/2c,
3a/3c, 4a/4c, 5a/5c, 1a/4a, 4a/2a, 1c/4c, 4c/2c, 3a/5a, 3c/5c,
0a/1a, 0a/5a}
\draw [-latex, line width=1pt, color=black] (\from) -- (\to);
\foreach \from/\to in {1a/3a, 1c/3c}
\draw [latex-latex, line width=1.5pt,color=black] (\from) -- (\to);
\end{tikzpicture}
\end{center}

Our proof is based on the following properties of conditional
probability, which remain true when replacing all ``$\ge$'' by
``$\le$''.

\begin{lmm}
\label{lmm-c1} Assume that $\mathscr{E}_{i}$ are disjoint events of
$X$ with $\sum_{i}\mathbb{P}(\mathscr{E}_{i})=1$. If
$(X|\mathscr{E}_{i})\ge Y$ for all $i$, then $X\ge Y$.
\end{lmm}

\begin{lmm}
\label{lmm-c2} Assume that $\mathscr{E}_{i},\mathscr{E}_{i}^{\prime
}$ are disjoint events of $X,Y$ with $\mathbb{P} (\mathscr{E}_{i})
=\mathbb{P}( \mathscr{E}_{i}^{\prime })$ and $\sum_{i} \mathbb{P}
(\mathscr{E}_{i})=1$. If $(X|\mathscr{E}_{i}) \ge
(Y|\mathscr{E}_{i}^{\prime })$ for all $i$, then $X\ge Y$.
\end{lmm}

We apply induction for all proofs. The initial conditions in all
cases can be readily checked. We assume that all the stochastic
dominance relations from (0a) to (5c) hold for all indices up to
$n-1$. We will then prove that they also hold when the indices are
$n$.

\paragraph{Proof of (0a), (0b).}
$1+B_{n}\ge A_{n},\ 1+A_{n}\ge B_{n}$.

We order each seat with a number from $1$ to $2n$ for
$\mathscr{A}_{n}$ and $\mathscr{B}_{n}$ as follows.

\begin{center}
\begin{tikzpicture}[scale=0.5, transform shape,
    every node/.style={circle,draw,minimum width =1.6cm,thick}]
\draw (-3,1) node[draw=none] {\Huge {$\mathscr{A}_{n}$}};
\node (a1) at (0, 2) {$2n$};
\node (a2) at (2, 2) {$2n-1$};
\node (a3) at (4, 2) {$2n-2$};
\node (a4) at (6, 2) {$2n-3$};

\node (a6) at (10, 2) {$n+4$};
\node (a7) at (12, 2) {$n+3$};
\node (a8) at (14, 2) {$n+2$};
\node (a9) at (16, 2) {$n+1$};

\node (b2) at (2, 0) {$n$};
\node (b3) at (4, 0) {$n-1$};
\node (b4) at (6, 0) {$n-2$};
\node (b5) at (8, 0) {$n-3$};
\node (b7) at (12, 0) {$4$};
\node (b8) at (14, 0) {$3$};
\node (b9) at (16, 0) {$2$};
\node (b10) at (18, 0) {$1$};

\fill (8,2) circle (2pt);
\fill (8.5,2) circle (2pt);
\fill (7.5,2) circle (2pt);

\fill (10,0) circle (2pt);
\fill (10.5,0) circle (2pt);
\fill (9.5,0) circle (2pt);

\end{tikzpicture}
\end{center}

and

\begin{center}
\begin{tikzpicture}[scale=0.5, transform shape]
\draw (-3,1) node {\Huge {$\mathscr{B}_{n}$}};
\tikzstyle{every node} = [circle,draw,minimum width =1.6cm,thick]
\node (a1) at (0, 0) {$2n$};
\node (a2) at (2, 2) {$2n-1$};
\node (a3) at (4, 2) {$2n-2$};
\node (a4) at (6, 2) {$2n-3$};

\node (a6) at (10, 2) {$n+4$};
\node (a7) at (12, 2) {$n+3$};
\node (a8) at (14, 2) {$n+2$};
\node (a9) at (16, 2) {$n+1$};

\node (b2) at (2, 0) {$n$};
\node (b3) at (4, 0) {$n-1$};
\node (b4) at (6, 0) {$n-2$};
\node (b5) at (8, 0) {$n-3$};
\node (b7) at (12, 0) {$4$};
\node (b8) at (14, 0) {$3$};
\node (b9) at (16, 0) {$2$};
\node (b10) at (18, 0) {$1$};

\fill (8,2) circle (2pt);
\fill (8.5,2) circle (2pt);
\fill (7.5,2) circle (2pt);

\fill (10,0) circle (2pt);
\fill (10.5,0) circle (2pt);
\fill (9.5,0) circle (2pt);
\end{tikzpicture}
\end{center}

Let $\mathscr{E}_{i},\mathscr{E}_{i}^{\prime }$ be the events of $
A_{n},B_{n} $ in which the first diner occupies seat number $i$.
Then
\[
    (A_{n}|\mathscr{E}_{i})\stackrel{d}{=}B_{n-i}+1+A_{i-2},
    \quad (B_{n}|\mathscr{E}_{i}^{\prime })\stackrel{d}{=}
    A_{n-i}+1+A_{i-2},
\]
for $1\le i\le n$,
\[
    (A_{n}|\mathscr{E}_{i})\stackrel{d}{=}A_{n-j-1}+1+B_{j-1},
    \quad (B_{n}|\mathscr{E}_{i}^{\prime })
    \stackrel{d}{=}B_{n-j-1}+1+B_{j-1},
\]
for $i=n+j,1\le j\le n-1$, and
\[
    (A_{n}|\mathscr{E}_{2n})\stackrel{d}{=}1+B_{n-1},
    \quad (B_{n}|\mathscr{E}_{2n}^{\prime })
    \stackrel{d}{=}1+A_{n-1}.
\]
By the induction hypothesis of (0a) and (0b),
\[
    (1+B_{n}|\mathscr{E}_{i}^{\prime })\ge (A_{n}|\mathscr{E}_{i})
    \quad\text{and}\quad (1+A_{n}|\mathscr{E}_{i})
    \ge (B_{n}|\mathscr{E}_{i}^{\prime })
\]
for $1\le i\le 2n$. By Lemma~\ref{lmm-c2}, we then prove the two
relations $1+B_{n}\ge A_{n}$ and $1+A_{n}\ge B_{n}$.

Note that the proof uses only relations between $A_\cdot$ and
$B_\cdot$; all other proofs will require the induction hypothesis
from other dominance relations.

\paragraph{Proof of (1a), (1b).} $A_{n}\ge A_{n-1},B_{n}\ge B_{n-1}$.

We first show that $A_{n}\ge A_{n-1}$. Let $\mathscr{E}_{1},
\mathscr{E} _{2}$ be the events of $A_{n}$ in which the first
customer selects seat number $1$ and $2$, respectively. Let
$\mathscr{E}_{c}$ be the event of $A_{n}$ in which the first
customer selects seat other than numbers $1,2$.
\[
    \begin{array}{cccccccccccc}
    1 & \bigcirc  & \bigcirc  & \bigcirc  & \bigcirc
    & \bigcirc  & \bigcirc  & \dots  & \bigcirc  & \bigcirc
    & \bigcirc  &  \\
    & 2 & \bigcirc  & \bigcirc  & \bigcirc  & \bigcirc
    & \bigcirc  & \bigcirc & \cdots  & \bigcirc  & \bigcirc
    & \bigcirc
    \end{array}
\]
To apply Lemma \ref{lmm-c1}, we need $(A_{n}|\mathscr{E}_{1}),
(A_{n}|\mathscr{E}_{2}),(A_{n}|\mathscr{E}_{c})\ge A_{n-1}$. We have
\[
    (A_{n}|\mathscr{E}_{1})\stackrel{d}{=}1+B_{n-1}\ge A_{n-1},
\]
by the induction hypothesis of (0a), and
\[
    (A_{n}|\mathscr{E}_{2})\stackrel{d}{=}B_{0}+1+A_{n-2}
    =2+A_{n-2}\ge 1+B_{n-1}\ge A_{n-1},
\]
by the induction hypothesis of (3b) and (0a). Thus
$(A_{n}|\mathscr{E} _{1}),(A_{n}|\mathscr{E}_{2})\ge A_{n-1}$.

To show that $(A_{n}|\mathscr{E}_{c})\ge A_{n-1}$, we consider
$(A_{n}|\mathscr{E}_{c})$ and $A_{n-1}$ (defined on the same
probability space) and apply Lemma~\ref{lmm-c2}. Let
$\mathscr{E}_{j}^{\prime}$ be an event of $A_{n-1}$ in which the
first customer sits on some seat. Similar to the proof of (0a) and
(0b), we see that
\[
    ((A_{n}|\mathscr{E}_{c})|\mathscr{E}_{j}^{\prime })
    \stackrel{d}{=}B_{n-k}+1+A_{k-2}\quad \text{and}
    \quad (A_{n-1}|\mathscr{E}_{j}^{\prime })
    \stackrel{d}{=}B_{n-k-1}+1+A_{k-2},
\]
for some $1\le k\le n-1$, or
\[
    ((A_{n}|\mathscr{E}_{c})|\mathscr{E}_{j}^{\prime })
    \stackrel{d}{=}A_{n-k-1}+1+B_{k-1}\quad \text{and}
    \quad (A_{n-1}|\mathscr{E}_{j}^{\prime })
    \stackrel{d}{=}A_{n-k-2}+1+B_{k-1},
\]
for some $1\le k\le n-1$. By induction hypothesis of (1a) and(1b),
\[
    ((A_{n}|\mathscr{E}_{c})|\mathscr{E}_{j}^{\prime })
    \ge (A_{n-1}|\mathscr{E}_{j}^{\prime })\quad \text{for all }j.
\]
By Lemma \ref{lmm-c2}, we obtain $(A_{n}|\mathscr{E}_{c})\ge
A_{n-1}$. This proves that $A_{n}\ge A_{n-1}$. The proof for
$B_{n}\ge B_{n-1}$ is similar.

The proofs for the other cases follow, \emph{mutatis mutandis}, the
same line of inductive arguments; details are straightforward and
omitted here.

\section{A combinatorial model}

Instead of the sequential stochastic model considered in this paper,
more static combinatorial models (sometimes referred to as hard-core
mode) were also considered in the literature, where all possible
unfriendly seating arrangements are equally likely. Such models turn
out to be much simpler to analyze. Let $N_n$ denote the total number
of distinct unfriendly seating arrangements under the initial
configuration $\mathscr{Y}_n$ (see \eqref{YYn}). Then $N_n$ is given
by the Fibonacci number
\[
    N_n = N_{n-1}+N_{n-2}\qquad(n\ge2),
\]
with $N_0=1$ and $N_1=1$. If we still denote by $X_n$ and $Y_n$ the
number of occupied seats when starting with the initial
configurations \eqref{XXn} and \eqref{YYn}, respectively, as we
studied above, then we have the simple recurrences for their
probability generating functions
\[
    X_n(t) = tY_{n-1}(t),\quad
    \text{and} \quad Y_n(t) = \frac{tN_{n-1}}{N_n}\, Y_{n-1}(t)
    +\frac{tN_{n-2}}{N_n}\, Y_{n-2}(t),
\]
with $Y_0(t)=1$ and $Y_1(t)=t$. This is easily solved and we have
\[
    X_n(t) = \frac{t}{N_{n-1}}
    \sum_{\cl{n/2}\le j\le n-1}\binom{j}{n-1-j} t^j
    \qquad(n\ge0),
\]
which is the essentially sequences A102426 and A098925 in Sloane's
Encyclopedia of Integer Sequences (see also A092865). This is also
connected to the number of parts in random compositions in which
only $1$ and $2$ are used. A local limit theorem with optimal
convergence rate can be derived by standard means;
see \cite[IX. 9]{FS09}.
The expected value is asymptotic to $\frac{2}{5-\sqrt{5}}\,n$ and
the variance to $\frac{3(3-\sqrt{5})}{\sqrt{5}(5-\sqrt{5})^2}\,n$.
Numerically, the jamming density is
\[
    \frac{1}{5-\sqrt{5}}\approx 0.36180\dots,
\]
which is smaller than that in the sequential model; the variance
constant is much smaller
\[
    \frac{3(3-\sqrt{5})}{\sqrt{5}(5-\sqrt{5})^2}\approx
    0.08944\dots.
\]
We conclude that the \emph{space utilization is better in the
sequential model than in the combinatorial model}. Such a property
has already been observed in the statistical physics literature; see
for example \cite{BK92} (where the combinatorial model is referred
to as the Hamiltonian system). Note that for the corresponding 1-row
seat configuration, one has the jamming density ($\alpha :=
\sqrt[3]{100+12\sqrt{69}}$)
\[
    \frac{(\alpha-2)(\alpha+2)^2(\alpha^3-192)}{4416}
    \approx 0.41149\dots,
\]
and the variance constant
\[
    \frac{6}{529}\cdot
    \frac{3\alpha^4+17\alpha^3-184\alpha^2+68\alpha+48}
    {(\alpha^2-2\alpha+4)^2}
    \approx 0.008539\dots.
\]
See \cite{Evans93,JM58} for more information.

%
%
%
%


\bibliographystyle{acm}

\end{document}